\documentclass[11pt]{article}
\usepackage[a4paper,margin=28mm]{geometry}
\usepackage{amsmath,amssymb,amsthm,mathtools,mathrsfs}
\usepackage{microtype}
\usepackage{enumitem}
\usepackage{graphicx}
\usepackage{float}
\usepackage[hidelinks]{hyperref}
\newtheorem{theorem}{Theorem}[section]
\newtheorem{proposition}[theorem]{Proposition}
\newtheorem{lemma}[theorem]{Lemma}
\newtheorem{corollary}[theorem]{Corollary}
\newtheorem{remark}[theorem]{Remark}
\newcommand{\Ric}{\operatorname{Ric}}
\newcommand{\Rm}{\operatorname{Rm}}
\newcommand{\RF}{\mathrm{RF}}
\newcommand{\PE}{\mathrm{PE}}

\title{\textbf{Quantized Einstein Metrics on \(S^7\) with \(SU(3)\)-Symmetry}}
\author{Anna Siffert}
\date{}

\begin{document}
\maketitle

\begin{abstract}
We prove that the standard cohomogeneity-one action of \(SU(3)\) on \(S^7\),
with principal orbit the Wallach flag manifold \(SU(3)/T^2\), admits
infinitely many invariant Einstein metrics.

The proof uses a detection approach to the Einstein boundary-value problem.
Starting from one singular orbit, we follow the Einstein solutions only to a
canonical hypersurface and measure there how far they are from closing
smoothly.  When the singular-orbit scale becomes small, this closing problem
is governed by a Ricci-flat limiting solution.  The linearisattion about its
limiting cone has an oscillatory mode.  As the scale shrinks, this oscillation
repeatedly changes the sign of the closing error, producing infinitely many
parameter values for which the metric closes.  The closing scales satisfy an asymptotic logarithmic quantization law,
with successive ratio tending to \(e^{-2\pi/\sqrt{15}}\).  We also determine
the geometry of the resulting sequence.  Away from the two singular orbits
the metrics converge to the singular sine cone over the non-normal Einstein
metric on \(SU(3)/T^2\), while after rescaling by the square of the
singular-orbit scale at either end they converge to the same complete
Ricci-flat threshold metric.  Their curvature is of order \(b_n^{-2}\), so
the logarithmic phase law induces a corresponding quantization of the focal
curvature scale.
\end{abstract}

\section{Introduction}

Einstein metrics on spheres form one of the basic testing grounds for the
interaction between topology, symmetry, and nonlinear geometric equations.
Even under a fixed cohomogeneity-one group action, however, the Einstein
equation is a singular two-point boundary-value problem, and multiplicity is
usually difficult to detect from the local initial data.  The purpose of this
paper is to show that the standard adjoint \(SU(3)\)-action on \(S^7\)
exhibits an infinite multiplicity phenomenon within one fixed symmetry
class.

Identify \(S^7\) with the unit sphere in \(\mathfrak{su}(3)\), with
\(SU(3)\) acting by the adjoint representation.  The orbit space is an
interval, a principal orbit is the six-dimensional Wallach flag manifold
\[
 SU(3)/T^2,
\]
and the two endpoint orbits are copies of \(\mathbb{CP}^2\).  Our main result
is the following.

\begin{theorem}[Main theorem]\label{thm:main}
For the standard cohomogeneity-one adjoint action of \(SU(3)\) on \(S^7\),
there exist infinitely many pairwise non-isometric smooth invariant Einstein
metrics.  They may be normalized by
\[
 \Ric(g_n)=6g_n,
\]
and have principal orbit \(SU(3)/T^2\) and two singular orbits
\(\mathbb{CP}^2\).

More precisely, the metrics can be chosen so that, if \(b_n\) denotes the
scale of the initial singular orbit and \(\beta_n\) its anisotropy parameter,
then
\[
 b_n\longrightarrow0,
 \qquad
 \beta_n\longrightarrow c_*<0,
\]
where \(c_*\) is the Ricci-flat threshold of
Proposition~\ref{prop:threshold}.  Moreover, the sequence may be indexed so that, for some phase constant
\(\theta_*\in\mathbb R/\pi\mathbb Z\),
\[
 \frac{\sqrt{15}}2\log\frac1{b_n}
 =n\pi+\theta_*+o(1),
\]
and consequently
\[
 \frac{b_{n+1}}{b_n}\longrightarrow e^{-2\pi/\sqrt{15}}.
\]
Let \(h_2\) denote the nonnormal Einstein metric on \(SU(3)/T^2\)
whose Ricci-flat cone is
\[
 dr^2+r^2h_2,
 \qquad
 f_1=f_2=\frac r{\sqrt5},\quad
 f_3=\sqrt{\frac25}\,r,
\]
and let \(g_{\RF,*}\) be the complete Ricci-flat threshold metric of
Proposition~\ref{prop:threshold}.  If \(p_n^\pm\) are points on the two
singular orbits, then
\[
 (S^7,b_n^{-2}g_n,p_n^\pm)
 \longrightarrow (M_*,g_{\RF,*},p_*)
\]
in pointed smooth Cheeger--Gromov sense.  On the complementary scale,
\[
 g_n\longrightarrow dt^2+\sin^2t\,h_2
\]
smoothly locally on \((0,\pi)\times SU(3)/T^2\).  Finally, there are
constants \(0<c<C<\infty\) such that
\[
 c b_n^{-2}\le \sup_{S^7}|\Rm(g_n)|\le C b_n^{-2}.
\]
Let \(K_{02}^{(n)}\) denotes the sectional curvature of \(g_n\) in the
two-plane spanned by the radial direction and a unit vector tangent to
the third isotropy summand. At eithe singular orbit,
$$
 b_n^2K_{02}^{(n)}\longrightarrow-\frac{1+2c_*^2}{2}.
$$
Consequently,
$$
 \frac{|K_{02}^{(n+1)}|}{|K_{02}^{(n)}|}
 \longrightarrow e^{4\pi/\sqrt{15}}.
$$
\end{theorem}

Thus the same complex cone exponent controls both the multiplicity of the
compact Einstein metrics and the geometric spacing of their bubbling
curvature scales.  Here \emph{bubbling} means that curvature concentrates near
the two singular orbits on the shrinking scale $b_n$, and that after rescaling
those regions by $b_n^{-2}$ one obtains a noncompact complete Ricci-flat
limit.  On the original scale the complementary region converges instead to
a singular Einstein sine cone.  The proof of the theorem occupies the
remainder of the paper.

\subsection{The boundary-value problem and the mechanism}

An invariant metric is determined along the orbit interval by three positive
metric functions.  The Einstein equation is therefore a singular nonlinear
ODE boundary-value problem, with smooth-collapse conditions at the two
singular orbits.  The local Einstein germs at one singular orbit form a
two-parameter family: we write $b>0$ for the size of that orbit and $\beta$
for its anisotropy parameter.

Rather than shooting all the way to the second singular orbit, we stop each
solution at the canonical hypersurface where the mean curvature $H$ first
vanishes.  Reflection across this orbit closes the half-metric precisely when
two remaining errors vanish.  These two errors form a detector
$\mathcal F=(F_1,F_2)$.  Thus the boundary-value problem is reduced to finding
zeros of a two-dimensional map.  Section~2 gives the invariant equations,
singular data, and the precise definition of this detector.

The source of infinitely many zeros appears as $b\downarrow0$.  Rescaling near
the shrinking singular orbit removes the positive Einstein constant and
produces a Ricci-flat limiting problem.  We identify a critical Ricci-flat
trajectory which approaches a distinguished cone.  Its linearization has one
relevant nonoscillatory mode and one complex mode.  The latter oscillates in
$\log b$.  We prove that this oscillatory mode is genuinely present and
survives transport back to the positive-Einstein problem.  Consequently one
closing error changes sign repeatedly while the other is controlled by the
second initial parameter, producing a simultaneous zero at infinitely many
scales.

This also explains why the sequence is difficult to see by direct shooting:
its scales accumulate geometrically at $b=0$, whereas in the logarithmic
variable $-\log b$ they become asymptotically periodic.  The detection and
response viewpoints of~\cite{SiffertDetection,SiffertResponse} provide useful
organizing language for this reduction, but all quantities and estimates
needed in the proof are defined explicitly below.

\paragraph{Numerical origin of the problem.}
The infinite family was first suggested by a simple shooting experiment.
Starting from the \(\mathbb{CP}^2\) singular orbit, we integrated the invariant
Einstein equations with initial parameters \((b,\beta)\) and stopped at the
first hypersurface \(H=0\).  We then adjusted \(\beta\) so that the even
closing error \(F_2=L_3\) vanished and followed the resulting numerical
branch as \(b\) decreased.  Along this branch the remaining odd closing error
\(F_1=f_1-f_2\) exhibited a persistent oscillation: after multiplication by
\(b^{-5/2}\), it was approximately periodic as a function of \(-\log b\); see
Figure~\ref{fig:numerical-response}.  This suggested an infinite sequence of
closing parameters accumulating at \(b=0\).  The proof below explains this
numerical observation rather than relying on it: after rescaling, the branch
is governed by a Ricci-flat threshold solution, and the odd linearization at
its limiting cone has the complex indicial roots
\[
 -\frac52\pm i\frac{\sqrt{15}}2.
\]
Thus the logarithmic oscillation seen numerically is precisely the response
predicted by the limiting cone.  No numerical data enter the proof.

\begin{figure}[H]
\centering
\includegraphics[width=.78\textwidth]{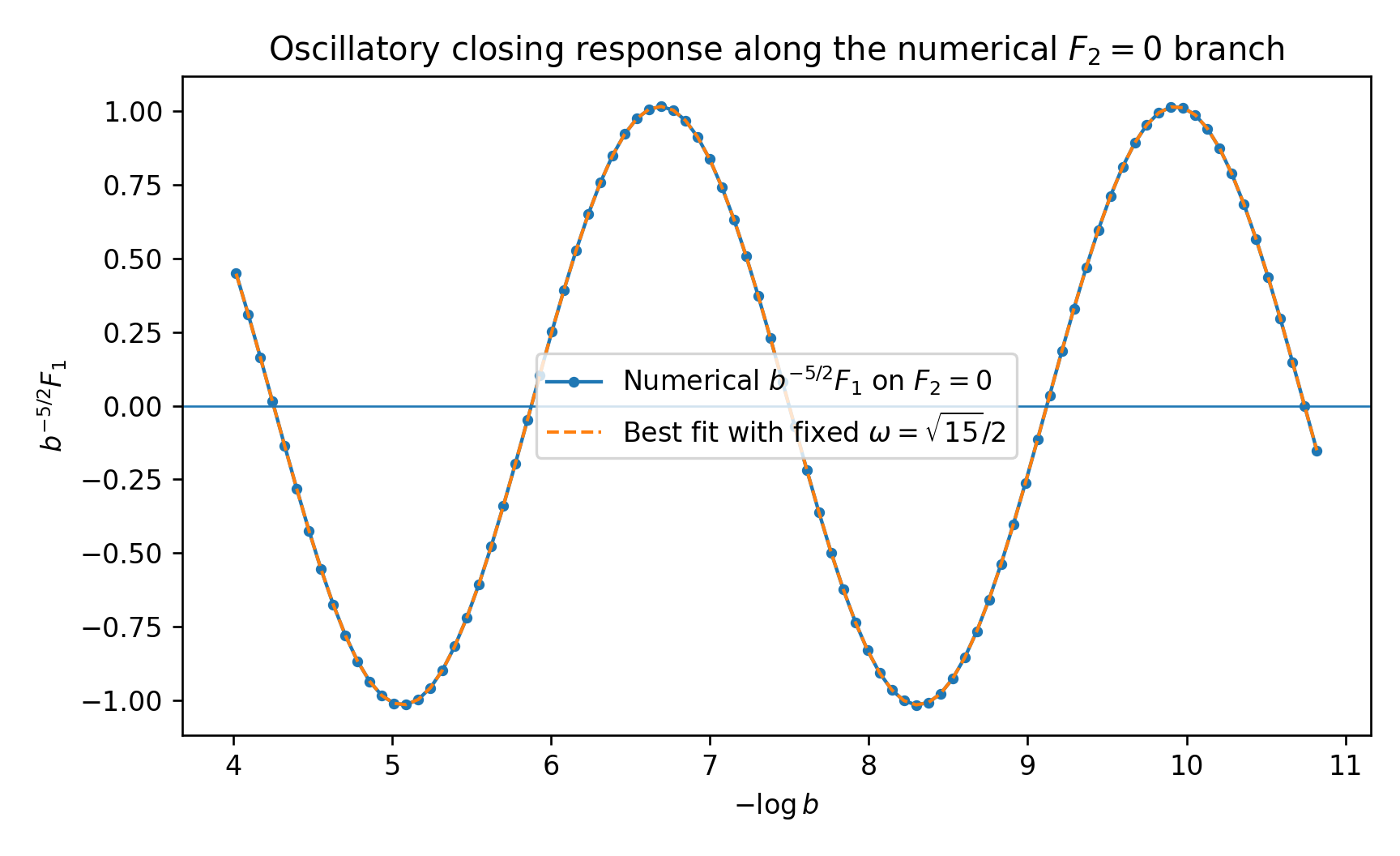}
\caption{Numerical motivation for the proof.  Along the numerically
determined near-closing branch \(F_2\approx0\), the normalized remaining
closing error \(b^{-5/2}F_1\) oscillates almost periodically as a function of
\(-\log b\).  The dashed curve is a best amplitude-and-phase fit with the
frequency fixed at the theoretically predicted value
\(\omega=\sqrt{15}/2\).  Successive zero crossings suggest the sequence of
Einstein metrics established analytically below.}
\label{fig:numerical-response}
\end{figure}

The result is specific to this symmetry class.  Böhm proved that \(S^7\)
carries infinitely many pairwise non-isometric positive Einstein
metrics~\cite{Bohm}; see also~\cite{BGK}.  Here we prove an infinite sequence
for the standard adjoint \(SU(3)\)-action itself, whose principal orbit is the
Wallach flag manifold \(SU(3)/T^2\).

Two earlier Wallach-space results are particularly relevant.  Chiu studies
cohomogeneity-one Einstein equations with Wallach principal orbits
numerically~\cite{Chiu}.  Chi's Ricci-flat analysis~\cite{Chi} enters in a
different role from the compact Einstein metrics constructed here.  When the
initial singular-orbit scale tends to zero, the rescaled positive-Einstein
equation loses its Einstein constant and converges to the Ricci-flat focal
problem; Chi's work provides the invariant-region input that we use on one
side of this limiting family.  Thus the Ricci-flat solutions are blow-up
models for the degeneration, not members of the compact family constructed
here.  The critical Ricci-flat trajectory needed here, its convergence to the
relevant cone, the nonvanishing oscillatory response, and the passage back to
compact positive-Einstein metrics are established below.

There is also recent work in which Ricci-flat cone dynamics is used to
organize cohomogeneity-one Einstein shooting.  Nienhaus--Wink
\cite{NienhausWink} prove the existence of three non-round Einstein metrics on
\(S^{10}\) by a phase-space analysis centered on a Ricci-flat trajectory, and
Huang \cite{Huang} uses precise winding information for Ricci-flat
trajectories; this yields an \(SO(3)\times SO(8)\)-invariant positive
Einstein metric on \(S^3\times S^7\).  These works are methodologically close in their
use of Ricci-flat limiting dynamics, but the mechanism needed here is
different in two respects: the limiting object is the Wallach threshold
selected below, and the infinite sequence is produced by a nonzero complex
odd response together with a two-component closing detector.  In particular,
the present argument does not infer compact solutions merely from winding of
a Ricci-flat trajectory.

The proof can therefore be read in four steps.  First we identify the
critical Ricci-flat trajectory obtained in the small-\(b\) limit.  Second we
analyze the two linearized modes at its limiting cone and prove that the
oscillatory one has nonzero amplitude.  Third we transport these two
variations back to the exact Einstein closing errors.  Finally their repeated
sign changes produce a simultaneous zero and hence a smooth closed metric.
Repeating the last step at successively smaller scales gives the infinite
sequence.

\paragraph{Organization of the paper.}
Section~2 sets up the invariant Einstein equations, singular initial data,
reflection detector, and phase--response reduction.  Section~3 constructs and
analyzes the critical Ricci-flat threshold.  Sections~4 and~5 identify the
linear responses at the limiting cone and prove that the oscillatory response
is nonzero.  Sections~6 and~7 transport these responses to the exact
positive-Einstein detector.  Section~8 closes the metrics by repeated phase
rectangles and proves the logarithmic quantization law.  Section~9 describes
the geometry of the resulting sequence, including the two Ricci-flat bubbles,
the singular sine-cone limit, and the curvature scale.

\subsection*{Acknowledgements}

The author used OpenAI's ChatGPT as an interactive assistant during the
preparation of this manuscript for exposition, organization, creating the pictures,
proofreading, and suggestions for presenting and checking arguments.
All these contributions were independently verified by the author,
who takes full responsibility for the mathematical content and
conclusions of the paper.

\section{The invariant Einstein problem}\label{sec:invariant-problem}

\subsection{Metric and singular data}

We now introduce the notation needed for the proof.  The isotropy
representation of \(SU(3)/T^2\) splits into three inequivalent real
two-dimensional summands
\[
 \mathfrak m_1\oplus\mathfrak m_2\oplus\mathfrak m_3.
\]
With a fixed invariant background form \(Q\), the metrics considered here are
\begin{equation}\label{eq:metric}
 g=dt^2+f_1(t)^2Q|_{\mathfrak m_1}
       +f_2(t)^2Q|_{\mathfrak m_2}
       +f_3(t)^2Q|_{\mathfrak m_3},
\end{equation}
The coefficient \(1\) of \(dt^2\) means that \(t\) is arclength along a
unit-speed normal geodesic.  This is the standard radial gauge, not an
additional homothety normalization.  We normalize the Einstein scale by
\[
 \Ric(g)=6g.
\]
This fixes the overall homothety scale of the metric.  It does not prescribe
the length of the cohomogeneity-one orbit interval: that length is determined
dynamically by the closing solution.

At the left singular orbit the first summand collapses.  In the unit-speed
radial gauge, smoothness across the singular orbit is expressed by the usual
parity conditions
\[
 f_1(-t)=-f_1(t),\qquad f_2(-t)=f_3(t).
\]
Thus the collapsing factor is odd with unit first derivative, while the two
noncollapsing factors are exchanged by the local Weyl symmetry.  The leading
singular data may therefore be written
\[
 f_2(0)=f_3(0)=b,\qquad
 f_2'(0)=-f_3'(0)=\beta,\qquad
 f_1'(0)=1.
\]
Here \(b>0\) is the size of the singular orbit and \(\beta\) measures its
anisotropy.  Once \(t\) is the unit-speed normal coordinate and \(Q\) is fixed,
the condition \(f_1'(0)=1\) is the smoothness condition for the collapsing
two-sphere direction; it is not a further freedom used to set a geometric
scale.  These are the two parameters that will eventually be adjusted
to satisfy the two closing conditions.

The initial point is singular, so this is not an ordinary regular ODE
problem at \(t=0\).  We use the local existence theory of
Eschenburg--Wang~\cite{EW}.  For completeness, we fix a Weyl-equivariant
diagonal branch of their singular construction.  On this branch the germ
depends real analytically on the leading data.  This follows from the
singular integral formulation in the proof of the Eschenburg--Wang initial
value theorem: after the recursive formal coefficients have been fixed, the
remainder is obtained as the fixed point of a contraction depending
analytically on the finite-dimensional initial data; the analytic
implicit-function theorem then gives analytic dependence of that fixed point.
We recall this argument at the only place where it is needed,
Lemma~\ref{lem:signed-a}.  No uniqueness among all possible singular germs is
required.

\subsection{Einstein equations and the closing detector}

Set
\[
 L_i=\frac{f_i'}{f_i},\qquad H=2(L_1+L_2+L_3).
\]
We fix the Einstein-equation convention once and for all.  If
\(L=\frac12g_t^{-1}\dot g_t\) and \(r_t\) denotes the Ricci endomorphism of
the principal-orbit metric \(g_t\), the radial and tangential Ricci formulas
are
\begin{equation}\label{eq:master-ricci}
 \Ric(\partial_t,\partial_t)=-H'-|L|^2,\qquad
 \Ric(g_t)^\sharp=r_t-L'-HL.
\end{equation}
Hence \(\Ric(g)=6g\) is equivalent, on the regular part, to
\begin{equation}\label{eq:master-einstein}
 H'=-|L|^2-6,\qquad
 L_i'=r_i-6-HL_i.
\end{equation}
For the Wallach metric \eqref{eq:metric}, the orbit-Ricci eigenvalues are
\begin{equation}\label{eq:intro-orbit-ricci}
 r_i=
 \frac1{f_i^2}
 +\frac{f_i^4-(f_j^2-f_k^2)^2}
 {4f_1^2f_2^2f_3^2},
 \qquad \{i,j,k\}=\{1,2,3\}.
\end{equation}
Taking the trace of the tangential equation and eliminating \(H'\) gives the
Hamiltonian constraint
\begin{equation}\label{eq:master-hamiltonian}
 \operatorname{Scal}(g_t)-H^2+|L|^2=30.
\end{equation}
Because each isotropy summand has real dimension two,
\[
 |L|^2=2(L_1^2+L_2^2+L_3^2).
\]
For the Ricci-flat inner problem the same formulas hold with every occurrence
of the Einstein constant \(6\) removed.  All later evolutionm and constraint
identities use this convention.

A smooth Weyl reflection across a principal orbit exchanges the first two metric directions and
fixes the third.  At a reflection orbit \(T\) its conditions are
\begin{equation}\label{eq:reflection}
 f_1(T)=f_2(T),\qquad
 L_1(T)+L_2(T)=0,\qquad
 L_3(T)=0.
\end{equation}
Since \(f_1(T)=f_2(T)>0\), the middle condition is equivalent to
\(f_1'(T)=-f_2'(T)\); likewise \(L_3(T)=0\) is equivalent to
\(f_3'(T)=0\).  Thus \eqref{eq:reflection} is exactly the \(C^1\) matching
condition for the Weyl reflection that exchanges the first two summands and
fixes the third.
By \eqref{eq:master-einstein},
\begin{equation}\label{eq:Hmono}
 H'=-2(L_1^2+L_2^2+L_3^2)-6<0.
\end{equation}
Thus \(H=0\) is a canonical transverse reflection section.  Its location is
not prescribed: for each focal solution the first zero occurs at a
parameter-dependent time \(T=T(b,\mu)\).  A closed solution obtained by
reflection therefore has orbit-space length \(2T(b,\mu)\), which is part of
the solution rather than a normalization imposed in advance.  On this moving
section the middle condition in \eqref{eq:reflection} is automatic.  We are
left with only
\[
 f_1-f_2=0,\qquad L_3=0.
\]

If \(T(b,\mu)\) denotes the first \(H=0\) time for the parametrization used
later, define
\begin{equation}\label{eq:detector-map}
 \mathcal F(b,\mu)
 =\bigl(F_1(b,\mu),F_2(b,\mu)\bigr)
 =\bigl(f_1-f_2,L_3\bigr)\big|_{t=T(b,\mu)}.
\end{equation}
A zero of \(\mathcal F\) is exactly a half-metric that can be reflected
smoothly to a closed metric on \(S^7\).  The rest of the paper is devoted to
showing that \(\mathcal F\) has infinitely many zeros.

\subsection{Phase--response reduction}

The final closing step is a finite-dimensional response-crossing argument,
viewed through the detection and response-geometric framework of
\cite{SiffertDetection,SiffertResponse}.  If one component of a
two-dimensional limiting detector selects an oscillatory phase transversely,
then the corresponding parameter direction is visible to first order.  One
may restrict the second observation to this phase branch, and a sign change of
the resulting reduced response persists under a sufficiently small uniform
perturbation of the detector.

Here the shrinking singular-orbit scale itself supplies the phase.  Write
\[
 \theta=\omega\log(1/b),\qquad
 b_{n,\theta}=\exp\!\left(-\frac{2\pi n+\theta}{\omega}\right),
\]
and normalize the detector by
\[
 \mathcal G_n(\theta,\mu)
 =
 \left(
 b_{n,\theta}^{-5/2}F_1(b_{n,\theta},\mu),
 F_2(b_{n,\theta},\mu)
 \right).
\]
The matching theorem proved below gives, uniformly on fixed compact
\((\theta,\mu)\)-sets,
\begin{equation}\label{eq:phase-response-limit}
 \mathcal G_n(\theta,\mu)
 \longrightarrow
 \left(
 \Re(C_o(\mu)Qe^{i\theta}),
 \mathcal E(\mu)
 \right).
\end{equation}
After shrinking the detuning interval, \(C_o(\mu)Q\) is bounded away from
zero.  The odd response therefore selects a phase branch
\(\theta_0(\mu)\) transversely, for example by
\[
 e^{i\theta_0(\mu)}C_o(\mu)Q=i|C_o(\mu)Q|.
\]
On this branch the reduced response is simply
\[
 \Psi(\mu)=\mathcal E(\mu).
\]
Since
\[
 \mathcal E(-\mu_0)<0<\mathcal E(\mu_0),
\]
the response-crossing principle applies to \(\mathcal G_n\) for every
sufficiently large \(n\).  Equivalently, one obtains the repeated
Poincar\'e--Miranda rectangles used in the closing argument.

This is the concrete phase--response mechanism in the present Wallach
problem.  Its conceptual interpretation comes from
\cite{SiffertDetection,SiffertResponse}, while the Wallach-specific content is
the limiting map~\eqref{eq:phase-response-limit}: threshold selection, Weyl
splitting, nonvanishing of the odd amplitude, and relative odd transport.
Thus response geometry supplies the organizing language, whereas the proof of
the required response crossing is entirely contained in the analysis below.

\section{The Ricci-flat threshold}

\paragraph{Time variables used below.}
To keep the several rescalings distinct, we use the following convention
throughout the threshold and matching arguments:
\[
\begin{array}{c|c}
\text{variable} & \text{meaning}\\ \hline
t & \text{physical unit-speed radial variable}\\
\tau=t/b & \text{inner variable after scaling by the focal size }b\\
s=\log\tau & \text{logarithmic inner variable}\\
\eta & \text{Ricci-flat projective time, }d\eta=H\,dt .
\end{array}
\]
A dot always denotes differentiation in physical time \(t\).  In the
projective calculation below, a prime denotes \(d/d\eta\) until physical time
is explicitly restored; in the later inner analysis \(d/ds\) is written
explicitly whenever confusion could arise.

When the singular orbit in the Einstein problem is very small, magnifying
that region removes the positive Einstein constant to first order.  The
resulting local model is Ricci-flat, and the relevant limiting objects form a
one-parameter Ricci-flat family issuing from the same singular orbit.

There are two qualitatively different kinds of Ricci-flat trajectories.
For one range of the initial anisotropy the metric is complete and enters a
controlled region of phase space.  For sufficiently large anisotropy in the
opposite direction the solution becomes incomplete.  Somewhere between
these two behaviours lies a critical initial value.  We call the
corresponding solution the \emph{threshold trajectory}.

Its long-time geometry is decisive: the threshold approaches the Ricci-flat
cone whose oscillatory mode drives the later closing argument.  Thus this
section proves
\[
 \text{complete/incomplete transition}
 \quad\Longrightarrow\quad
 \text{a distinguished conical limit}.
\]
The oscillation around that limit will be analyzed only afterwards.

A boundary-layer argument produces an open tail of incomplete Ricci-flat
initial data, while Chi's invariant region supplies an open family of complete
initial data.  Their boundary selects the threshold.  Finite-time termination
and escape in the normalized phase variables are then excluded, and a
monotone quantity forces convergence to the distinguished cone.

In the Ricci-flat limiting problem, which has its own homothety freedom,
we fix the noncollapsing singular-orbit scale.  This is only a normalization
of the inner Ricci-flat model and does not impose a length on the
positive-Einstein orbit interval.  Thus only one Ricci-flat parameter remains.  We denote it by
\[
 c:=\beta.
\]
With the sign convention used here, the incomplete tail is
\(c\to-\infty\).

\subsection{Producing an incomplete side}

We first need only one robust fact: sufficiently negative \(c\) cannot give
a complete Ricci-flat metric.  The detailed boundary-layer variables below
are a convenient way to prove this.  They magnify the short initial interval
on which a large anisotropy dominates the evolution; they are not part of
the later Einstein matching construction.

For large anisotropy write \(c=-\varepsilon^{-1}\) and rescale
\[
 t=\varepsilon\tau,\qquad f_1=\varepsilon A,\qquad f_2=B,\qquad f_3=C.
\]
The boundary-layer limit has the first integral $BC\equiv1$.  With
\[
 \delta=\log(C/B),\qquad u=A',\qquad w=A\delta',
\]
one obtains
\begin{equation}\label{eq:bl}
 u'=-\frac{w^2}{2A},\qquad
 \delta'=\frac{w}{A},\qquad
 w'=\frac{\sinh(2\delta)-uw}{A},
\end{equation}
with constraint
\begin{equation}\label{eq:blconstraint}
 u^2-\frac12w^2=1-\sinh^2\delta.
\end{equation}
On the positive branch, $w>0$, hence $\delta'>0$ and $u'<0$.  If $u$ stayed
nonnegative, then once $\delta\ge\operatorname{arsinh}1$,
\[
 \frac{du}{d\delta}=-\frac w2
 \le-\frac1{\sqrt2}\sqrt{\sinh^2\delta-1},
\]
which forces $u$ through zero if $\delta\to\infty$.  If instead $\delta$ remained bounded while $u\ge0$, then $u$ and $w$ would
remain bounded by the constraint and $A'=u$ would give $A(\tau)\le C(1+\tau)$.
But the first and third equations in \eqref{eq:bl} imply the exact identity
\[
 (Aw)'=A'w+Aw'=\sinh(2\delta).
\]
Since $\delta$ is increasing and positive after the focal end, boundedness
would give $\delta\to\delta_\infty>0$ and hence $Aw\ge c\tau$ for large
$\tau$.  Thus $w\ge c_0>0$, and then
$\delta'=w/A\ge c_1/(1+\tau)$, contradicting boundedness of $\delta$.
Hence $u$ becomes negative at finite rescaled time.  Choose a regular rescaled
time at which $u<0$.  Smooth dependence of the rescaled singular germ on
$\varepsilon$ then gives $H<0$ at a regular point for all sufficiently small
$\varepsilon>0$.  But
\[
 H'=-2(L_1^2+L_2^2+L_3^2)\le-\frac{H^2}{6},
\]
so they are incomplete in finite forward time.

\begin{proposition}\label{prop:incomplete}
The Ricci-flat focal family has an incomplete tail for sufficiently large
anisotropy on the chosen side.
\end{proposition}

\subsection{The finite endpoint and its limit}

We now place the incomplete tail next to a known complete family.  This
creates the threshold parameter $c_*$.  The rest of the section is devoted
to showing that its trajectory does not terminate at a finite projective
boundary point but converges to a single interior equilibrium.

For the $SU(3)/T^2$ Wallach diagram, Chi constructs a continuous
one-parameter family of smooth complete asymptotically conical Ricci-flat
metrics on $\Lambda_-^2\mathbb{CP}^2$, with the Bryant--Salamon metric in the
interior~\cite{BryantSalamon,Chi}.  To keep the logical dependence transparent, we separate
the imported input from the argument proved here.

To state that input, we retain Chi's normalized polynomial variables
\(X_i,Z_i\).  The \(X_i\) are normalized logarithmic-derivative variables and
the \(Z_i>0\) are the corresponding curvature variables; below we need only
Chi's polynomial identities and ratios of the \(Z_i\), not their full
normalization.  Set
\[
 d_1=X_1-X_3,\qquad d_2=X_2-X_3.
\]
We also retain Chi's notation \(S_3\) for his Case-I semialgebraic region.
The smaller invariant set \(\widehat S_3\) is defined explicitly at its first
appearance.

\medskip
\noindent\textit{Input from Chi.}
We use only the following Case-I facts from~\cite{Chi}.  Proposition~3.4
proves that, on the $s_1^{\mathrm{Chi}}\ge0$ side of the focal family, trajectories do not
leave the region $S_3$.  Lemma~3.6 shows that
\[
 \widehat S_3
 =S_3\cap\{Z_1+Z_2-Z_3\ge0\}
 \cap\{Z_1(X_1-X_3)+Z_2(X_2-X_3)\ge0\}
\]
is compact and forward invariant.  More precisely, Lemma~3.19 of~\cite{Chi}
constructs an entrance zone from which sufficiently small positive \(s_1^{\mathrm{Chi}}\)
trajectories are forced to exit through the face \(Z_1+Z_2-Z_3=0\);
Lemma~3.22 then shows that the exit point lies on
\(\partial\widehat S_3\), is noncritical, and hence enters the compact
invariant set.  Thus a nontrivial interval of focal parameters exists for all
forward time.  We do \emph{not} import from~\cite{Chi} the existence,
uniqueness, asymptotics, or cone limit of the far endpoint introduced below.

\medskip
\noindent\textit{New threshold argument.}
The relation with Chi's parameter is explicit.  His smooth focal data are
\[
 (f_1,f_2,f_3,f_1',f_2',f_3')(0)
 =(0,h_0,h_0,1,-h_1,h_1),
\]
whereas our convention is $f_2'(0)=-f_3'(0)=c$.  Hence
$c=-h_1$.  Moreover, Chi's equation~(2.27) gives, for fixed
\(s_0^{\mathrm{Chi}}>0\) and with \(d=2\), the common dimension of the three
isotropy summands,
\[
 \frac{2h_1}{\sqrt d}
 =\frac{\sqrt2\,s_1^{\mathrm{Chi}}}{\sqrt{(d+1)s_0^{\mathrm{Chi}}}},
\]
so \(c\) and \(s_1^{\mathrm{Chi}}\) have opposite signs.  Therefore the incomplete side
$c<0$ is exactly Chi's $s_1^{\mathrm{Chi}}>0$ side (after the harmless interchange of the
second and third summands if the opposite convention is chosen).  In
particular Proposition~3.4 applies on the entire side used below, not merely
after an unspecified reparametrization.  Since the far endpoint below lies
strictly away from $c=0$, all sufficiently nearby parameters on either side
of it remain in this same $s_1^{\mathrm{Chi}}>0$ half-family.  Let \(I_{\mathrm{ent}}\) be the connected component, containing the
Bryant--Salamon value \(c=0\), of the set of parameters on the side \(c\le0\)
whose trajectories either start in \(\widehat S_3\) (at \(c=0\)) or enter
\(\widehat S_3\) through its noncritical entrance face in finite projective
time.  Lemmas~3.19 and~3.22 of~\cite{Chi}, together with the positive
proportionality between \(h_1\) and \(s_1^{\mathrm{Chi}}\), show that
\(I_{\mathrm{ent}}\) contains a nontrivial one-sided neighborhood of
\(c=0\).
Proposition~\ref{prop:incomplete} shows that \(I_{\mathrm{ent}}\) cannot extend
indefinitely along the chosen anisotropic side.  Since it contains a
nontrivial one-sided neighborhood of \(0\), its far endpoint is a finite
number \(c_*<0\).  By definition, parameters of \(I_{\mathrm{ent}}\) approach
\(c_*\) from the entrance side, while parameters immediately beyond \(c_*\)
on the same half-family do not belong to this entrance component.  Local existence of the focal solutions follows from the singular-orbit
initial-value theory of Eschenburg--Wang~\cite{EW}.  Their construction also
exhibits the possible higher-order indeterminacy; throughout we fix the
Weyl-equivariant recursive complement branch described above.  On that fixed
finite-dimensional branch, the singular integral formulation gives the
continuous, and in fact real-analytic, parameter dependence, used below.  In the present
group diagram the tangent and normal slice representations have no common
irreducible summand.

We now identify the limiting trajectory at $c_*$.  There are two possible
obstructions to convergence to an interior equilibrium: the trajectory might
hit the entrance face at finite time, or it might run toward the projective
boundary where the reconstruction factor $Z_3$ diverges.  We exclude these
possibilities in that order.

Pass to Chi's projective time and set
\[
 \omega_1=\frac{Z_1}{Z_3},\qquad
 \omega_2=\frac{Z_2}{Z_3},\qquad
 S=\omega_1+\omega_2.
\]
Until physical time is explicitly restored below, a prime in this threshold
argument denotes differentiation in projective time.  The state coordinates
used in the contact calculation are
\((\omega_1,\omega_2,d_1,d_2)\).
The entrance face is $S=1$.  The exact projective equations give
\begin{equation}\label{eq:Scontact}
 S=1,\quad S'=0
 \quad\Longrightarrow\quad
 S''=4(\omega_1d_1^2+\omega_2d_2^2)\ge0.
\end{equation}
If equality holds, then $d_1=d_2=0$.  At such a point write
$\omega_2=1-\omega_1$.  Differentiating the exact projective equations gives
\[
 S'''=0,
 \qquad
 S''''=8Z_3^4\,\omega_1(1-\omega_1)(2\omega_1-1)^2.
\]
For an interior contact $0<\omega_1<1$, this is strictly positive unless
$\omega_1=1/2$.  Hence the only interior degenerate contact is
\begin{equation}\label{eq:p2projective}
 p_2=\left(\frac12,\frac12,0,0\right).
\end{equation}

Figure~\ref{fig:threshold-schematic} summarizes the threshold geometry in the
\((\omega_1,\omega_2)\)-projection.  It is schematic rather than a numerical
phase portrait.  The two neighboring families are drawn on opposite sides of
the threshold only to visualize how their first-contact points on \(S=1\)
approach the distinguished point \(p_2\).  The figure is not used as input
to Proposition~\ref{prop:threshold} or to any of the estimates below.

\begin{figure}[ht]
\centering
\includegraphics[width=.76\textwidth]{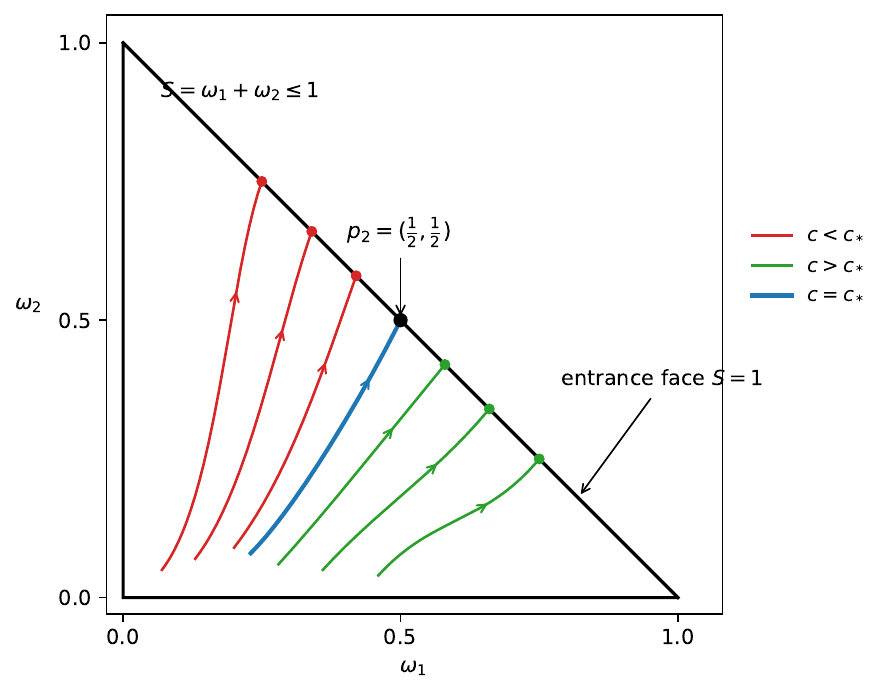}
\caption{Schematic threshold selection in projective variables.  The triangle
is the region \(S=\omega_1+\omega_2\le1\), and its sloping edge is the
entrance face \(S=1\).  The red and green curves schematically represent
parameters on the two sides of the threshold, with first-contact points on
\(S=1\) tending toward \(p_2=(1/2,1/2)\).  The blue curve represents the
threshold trajectory \(c=c_*\), which remains in \(S<1\) for finite
projective time and converges to \(p_2\).  The curves are schematic and are
not numerical trajectories.}
\label{fig:threshold-schematic}
\end{figure}

Let $c_j\to c_*$ through $I_{\mathrm{ent}}$ and let $T_j$ be the first
$S=1$ time.  We claim
\begin{equation}\label{eq:Ttoinfty}
 T_j\longrightarrow\infty.
\end{equation}
Suppose instead that $T_j\to T_*<\infty$ along a subsequence.  Smooth
dependence gives
\[
 S_{c_*}(T_*)=1,\qquad S_{c_*}(t)<1\quad(t<T_*).
\]
All $Z_i$ are positive at every positive regular time on the focal branch,
so this first contact is interior: $0<\omega_i<1$.  A nontransverse first
contact is impossible.  Indeed, since \(S<1\) immediately before the first
contact, a tangential contact cannot have \(S''>0\): that would make the
contact a local minimum.  Thus \eqref{eq:Scontact} forces \(S''=0\), hence
\(d_1=d_2=0\).  The same one-sided argument applied to the first possible
nonzero even derivative rules out \(S''''>0\).  The fourth-order calculation
therefore forces any such interior contact to be \(p_2\), while uniqueness for
the autonomous projective system prevents a nonconstant trajectory from
reaching the equilibrium \(p_2\) in finite time.  Hence
\[
 S'_{c_*}(T_*)>0.
\]

Now put
\[
 W=Z_1+Z_2-Z_3,\qquad
 \Xi=Z_1(X_1-X_3)+Z_2(X_2-X_3).
\]
At this point we use one further identity from Chi's polynomial
system.  Directly differentiating \(W=Z_1+Z_2-Z_3\) with the displayed
\(Z_i\)-equations gives, in polynomial time,
\[
 W'=2\Xi\qquad\text{on }W=0.
\]
Since $S-1=W/Z_3$ with $Z_3>0$, $S'$ and $\Xi$ have the same sign there.  Thus the limiting first contact satisfies $\Xi>0$.
By Proposition~3.4 of~\cite{Chi} the trajectory on the chosen side remains
in $S_3$ up to the contact, so
\[
 \gamma_{c_*}(T_*)\in
 S_3\cap\{W=0\}\cap\{\Xi>0\}\subset\widehat S_3.
\]
Because $W'>0$, this is a noncritical entrance point.  Lemma~3.6 of~\cite{Chi} then traps the trajectory in the compact set
$\widehat S_3$ for all later time.

This entrance persists for nearby focal parameters.  The implicit-function
theorem gives nearby first-contact times $T(c)$, and the strict inequality
$\Xi(\gamma_c(T(c)))>0$ remains valid.  Proposition~3.4 of~\cite{Chi} keeps all such trajectories in $S_3$ up to
their contact.  Hence parameters on both sides of
$c_*$ sufficiently close to it also enter $\widehat S_3$, contradicting the
maximality of $I_{\mathrm{ent}}$.  This proves \eqref{eq:Ttoinfty}.

\medskip
\noindent\textit{From here to Proposition~\ref{prop:threshold}, the argument is
new.}  Chi's invariant region is used only as a region in which the displayed
sign information is valid; the endpoint continuation, large-\(Z_3\) ejection,
precompactness, Lyapunov argument, and convergence to \(p_2\) are proved
below.

We next prove that the threshold is global and converges to the nonnormal cone.
The point requiring care is that the Hamiltonian constraint does \emph{not}
make the region \(K\ge0\) forward invariant.  We therefore use physical-time
continuation first and then a direct large-\(Z_3\) ejection estimate.

Let
\[
 H=\operatorname{tr}L=2(L_1+L_2+L_3).
\]
For the Ricci-flat equations,
\begin{equation}\label{eq:raychaudhuri-rf}
 \dot H=-2\sum_{i=1}^3L_i^2.
\end{equation}
If a regular solution has \(H>0\) on a finite physical interval
\([t_0,T)\), then \eqref{eq:raychaudhuri-rf} gives
\(\int_{t_0}^T\sum L_i^2<\infty\).  Hence
\(\int|L_i|<\infty\), the \(f_i\) stay bounded above and away from zero, the
orbit Ricci terms stay bounded, and
\[
 \dot L_i=-HL_i+r_i(f)
\]
keeps \(L_i\) bounded.  The solution therefore extends across \(T\).

For the threshold trajectory \(H\) cannot vanish at a finite regular time.
Indeed, if some \(L_i\neq0\), then \(\dot H<0\) there.  If all \(L_i=0\),
then the Ricci-flat equations give
\[
 \dot L_i=r_i,\qquad
 \dot H=\ddot H=0,\qquad
 H^{(3)}=-4\sum_{i=1}^3 r_i^2<0.
\]
The final inequality holds because otherwise the homogeneous principal orbit
would be Ricci-flat; a compact homogeneous Ricci-flat manifold is flat by
Alekseevskii--Kimel'fel'd~\cite{AK}, which \(SU(3)/T^2\) is not.  Thus a finite zero of \(H\) would be followed by a
regular time \(t_1\) at which \(H(t_1)=-a<0\).
Cauchy--Schwarz then gives
\[
 \dot H\le-\frac{H^2}{6},
\]
so comparison with \(\dot h=-h^2/6\) forces blow-down no later than
\(t_1+6/a\).  Choose a regular time \(t_2\in(t_1,t_1+6/a)\) sufficiently
close to the comparison blow-down time that \(H(t_2)<-M\), with \(M\) fixed
large.  Smooth dependence on the focal parameter gives the same strict
negative bound at \(t_2\) for all sufficiently close entrance-side
parameters.  The same comparison then forces those nearby trajectories to
become singular within physical time \(6/M\) after \(t_2\).  This contradicts
their forward entrance into Chi's compact invariant set (and hence their
completeness).  Therefore
\begin{equation}\label{eq:H-positive-threshold}
 H(t)>0\qquad\text{at every finite threshold time.}
\end{equation}

We now denote the projective time by \(\eta\); it satisfies
\(d\eta=H\,dt\).  On Chi's region \(S_3\), the normalized third orbit-Ricci component
\[
 \mathcal R_3:=\frac{r_3}{H^2}
\]
satisfies \(\mathcal R_3\ge0\).  Since on the \(s_1^{\mathrm{Chi}}>0\) branch \(L_3>0\) initially,
\[
 \dot L_3=-HL_3+r_3,\qquad r_3=H^2\mathcal R_3\ge0.
\]
If \(\int_0^\infty H\,dt<\infty\), the integrating factor would give
\(L_3\ge c>0\), while \eqref{eq:raychaudhuri-rf} would force \(H\) to become
negative in finite time.  Therefore
\begin{equation}\label{eq:eta-infinite}
 \int_0^\infty H\,dt=\infty,\qquad \eta\longrightarrow\infty.
\end{equation}

We now exclude escape in the normalized variables.  Put
\[
 R=Z_3,\qquad s=d_1+d_2,\qquad
 q=d_1^2-d_1d_2+d_2^2.
\]
The quadratic shear form used in the normalized equations is
\(\mathcal Q(v):=8q\).  In the Hamiltonian decomposition write
\[
 \mathcal G:=\frac16+\frac43q,\qquad
 K:=\frac56-\frac43q.
\]
The exact equations give
\begin{equation}\label{eq:R-log-derivative}
 \frac{R'}R=\frac23(2q-s)\ge-\frac13.
\end{equation}
Define the projective curvature polynomial
\[
 \mathcal D(\omega_1,\omega_2)
 =-S^2+8\omega_1\omega_2+6S-1.
\]
The Ricci-flat Hamiltonian constraint in Chi's normalized variables is
\begin{equation}\label{eq:D-constraint-new}
 \mathcal D(\omega_1,\omega_2)=\frac{5-8q}{3R^2}.
\end{equation}
Equivalently, with \(K=\frac56-\frac43q\),
\[
 \mathcal D=\frac{2K}{R^2}.
\]
This is a separate projective normalization of the Ricci-flat system and is
used below only through these displayed identities.
Moreover \(X_3\ge0\) on \(S_3\), hence
\begin{equation}\label{eq:s-half}
 s\le\frac12.
\end{equation}
In Chi's normalized polynomial variables, the corresponding exact
curvature-difference polynomials are
\[
 N_1=(\omega_1-1)(\omega_1-3\omega_2+1),\qquad
 N_2=(1-\omega_2)(3\omega_1-\omega_2-1).
\]
These are the projectivized form of the same Wallach curvature differences;
their normalization is Chi's \(Z_i\)-normalization fixed above, whereas
\eqref{eq:intro-orbit-ricci} is written directly in the metric coefficients.
No identification of the individual bnormalizing factors is used below.  The
polynomials satisfy
\[
 N_1+N_2=9S-3-\mathcal D.
\]
Substitution in those difference equations yields
\begin{equation}\label{eq:s-exact}
 s'=\frac{R^2}{2}\bigl(9S-3-\mathcal D(1+s)\bigr).
\end{equation}

\begin{lemma}[Large-\(R\) ejection]\label{lem:finite-exit}
Along every threshold segment with \(R\ge2\),
\begin{equation}\label{eq:s-drift}
 s'\le-\frac14R^2.
\end{equation}
Consequently the threshold satisfies
\begin{equation}\label{eq:Z3bound}
 Z_3=R<4.
\end{equation}
\end{lemma}

\begin{proof}
If \(\mathcal D<0\), then
\[
 \mathcal D=-S^2+8\omega_1\omega_2+6S-1
 \ge-S^2+6S-1.
\]
Thus \(S<3-2\sqrt2\).  If \(s<-1\), then
\(-\mathcal D(1+s)\le0\), and the bracket in
\eqref{eq:s-exact} is at most
\(9(3-2\sqrt2)-3<-1/2\).  If \(-1\le s\le1/2\), then
\[
 -\mathcal D(1+s)
 \le\frac32(S^2-6S+1),
\]
and the same bracket is again \(<-1/2\).

If \(\mathcal D\ge0\), then \(K\ge0\), so \(q\le5/8\) and
\(|s|\le2\sqrt q\le\sqrt{5/2}\).  For \(R\ge2\),
\[
 0\le\mathcal D=\frac{2K}{R^2}\le\frac5{12}.
\]
Since \(\mathcal D\ge-S^2+6S-1\) and
\(-1/16+3/2-1=7/16>5/12\), we have \(S<1/4\).
If \(s\ge-1\), the bracket in \eqref{eq:s-exact} is \(<-3/4\).
If \(s<-1\), then
\[
 -\mathcal D(1+s)
 \le\frac5{12}\left(\sqrt{\frac52}-1\right),
\]
so the bracket is still \(<-1/2\).  This proves
\eqref{eq:s-drift}.

Suppose now that \(R(\eta_0)=4\).  By
\eqref{eq:R-log-derivative}, before \(R\) could fall to \(2\),
\[
 R(\eta)\ge4e^{-(\eta-\eta_0)/3}.
\]
Combining this with \eqref{eq:s-drift}, the decrease of \(s\) before a
possible drop \(4\to2\) is at least
\[
 \frac14\cdot\frac32(4^2-2^2)=\frac92.
\]
Since \(s(\eta_0)\le1/2\), the trajectory reaches \(s=-2\) while \(R\ge2\).
There \(q\ge s^2/4\ge1\), hence \(K<0\), \(\mathcal D<0\), and
\(N_1+N_2<0\).  Equation \eqref{eq:s-exact} then gives
\[
 s'\le\left(\frac43q-\frac56\right)s
 \le\frac18s^3.
\]
Thus \(s\to-\infty\) at a finite projective time \(\eta_*<\infty\).
But \(s=d_1+d_2\) is a smooth normalized function of the regular Ricci-flat
state as long as \(H>0\).  The physical-time continuation proved above and
\eqref{eq:H-positive-threshold} therefore exclude such a finite-\(\eta\)
blow-up.  Equivalently, the global threshold has projective time
\(\eta\to\infty\) by \eqref{eq:eta-infinite}.  Hence \(R<4\).
\end{proof}

Since \(Z_1,Z_2\le Z_3<4\) and \(\omega\) stays in the compact triangle
\(S\le1\), the polynomial \(\mathcal D\) is bounded.  Therefore
\(K=\frac12R^2\mathcal D\) is bounded, and
\[
 q=\frac34\left(\frac56-K\right)
\]
is bounded.  Thus \(d_i,X_i,Z_i\) remain bounded, so the normalized threshold
is precompact for all \(\eta\ge0\).

Finally define
\begin{equation}\label{eq:Lyapunov}
 \mathscr L=\frac12R^2(\omega_1\omega_2)^{2/3}.
\end{equation}
Direct differentiation gives
\begin{equation}\label{eq:Lyapunov-derivative}
 (\log\mathscr L)'=\frac13\mathcal Q(v)\ge0.
\end{equation}
Precompactness bounds \(\mathscr L\) above, while positivity at one regular
time bounds it away from zero.  Thus for some \(c_0>0\) and all sufficiently
large \(\eta\),
\[
 \mathscr L(\eta)\ge c_0.
\]
Since \(R<4\) and \(0<\omega_i\le1\), this lower bound shows that, along the
tail, both \(R\) and \(\omega_1\omega_2\) are bounded away from zero.  In
particular every omega-limit point lies in the interior normalized chart.
Moreover
\[
 \int_0^\infty\mathcal Q(v)\,d\eta<\infty.
\]
The normalized orbit is precompact and the vector field is smooth there, so
\(\mathcal Q(v(\eta))\) is uniformly continuous.  Therefore
\(\mathcal Q(v(\eta))\to0\).  Since
\[
 d_1^2-d_1d_2+d_2^2
 =\frac12(d_1^2+d_2^2)+\frac12(d_1-d_2)^2,
\]
we obtain
\[
 d_1(\eta),d_2(\eta)\longrightarrow0.
\]
Let \(\Omega\) be the omega-limit set.  Precompactness makes \(\Omega\)
nonempty, compact, and invariant, and the preceding limit gives
\(\Omega\subset\{d_1=d_2=0\}\).  Hence the complete orbit through every point
of \(\Omega\) remains in this set, so the \(d_1\)- and \(d_2\)-components of
the normalized vector field vanish on \(\Omega\).  Because the normalization
factors in the exact difference equations are nonzero on the interior chart,
it follows that
\[
 N_1=N_2=0
\]
at every point of \(\Omega\).  Their common zeros are
\[
 \left(\frac12,\frac12\right),\quad(1,1),\quad(1,2),\quad(2,1).
\]
The inequality \(S=\omega_1+\omega_2\le1\) leaves only
\[
 (\omega_1,\omega_2)=\left(\frac12,\frac12\right).
\]
At this point \(q=0\) and \(\mathcal D=6\), so the Hamiltonian constraint
\eqref{eq:D-constraint-new} gives
\[
 R^2=\frac5{18}.
\]
Thus \(\Omega=\{p_2\}\).  A precompact trajectory with singleton omega-limit
set converges to that point, and therefore the threshold dconverges to \(p_2\).

\begin{proposition}[Threshold selection]\label{prop:threshold}
The endpoint Ricci-flat trajectory is global and converges to $p_2$.
\end{proposition}

\section{From the threshold cone to observable responses}

Section~3 identifies the inner limiting object: the Ricci-flat threshold
approaches one specific cone.  The next step in the response architecture is
to determine which perturbations of that cone are visible to the two
components of the reflection detector.

Two issues must be separated: the evolution of a small perturbation away
from the cone and its visibility in the reflection conditions.  At the
linearized sine-cone level this means nonvanishing at the model event
\(T_0\); for nearby solutions the detector is evaluated at their moving
first-\(H=0\) events.

Symmetry separates the answer into two simple pieces.  One perturbation
changes the first two metric factors in opposite directions.  It is the
oscillating mode, and the first closing error \(F_1=f_1-f_2\) sees it.  The
other perturbation preserves \(f_1=f_2\) and changes the third factor relative
to them.  It does not oscillate, and the second closing error \(F_2=L_3\)
sees it.  Thus the cone provides exactly the two responses needed later:
one response can be tuned across zero, while the other repeatedly changes
sign as the small scale varies.

The purpose of this section is to prove these two statements at the
linearized level.  The next section then verifies one essential
nondegeneracy: the threshold itself really has a nonzero component in the
oscillating direction.

From the projective reconstruction,
\[
 p_2:\qquad f_1=f_2,\qquad f_3=\sqrt2\,f_1.
\]
The corresponding Ricci-flat cone is
\begin{equation}\label{eq:cone}
 f_1=f_2=\frac r{\sqrt5},\qquad
 f_3=\sqrt{\frac25}\,r.
\end{equation}
Its link $h_2$ satisfies $\Ric(h_2)=5h_2$, so the sine cone
\[
 dt^2+\sin^2t\,h_2
\]
has Einstein constant \(6\).  On this model \(H=6\cot t\).  For this
limiting sine-cone model only, its first \(H=0\) section occurs at
\(t=\pi/2\), where
\(f_1=f_2\) and \(L_3=0\).  This value is a normalization of the limiting
model, not a prescribed reflection time for the nearby Einstein metrics.

\subsection{The oscillating response}

The first closing condition asks whether the first two scale factors agree.
It is therefore natural to measure their relative separation by
\[
 z=\log\frac{f_1}{f_2}.
\]
Linearizing at the sine cone gives the model equation
\begin{equation}\label{eq:odd-sine}
 z''+6\cot t\,z'+10\csc^2t\,z=0.
\end{equation}
Here and below the coefficients \(\cot t,\csc^2t\) are exact only on the
sine-cone background; nearby outer solutions are treated as variable-coefficient
perturbations of this model.
The indicial roots at the singular end are
\[
 \alpha_\pm=-\frac52\pm i\omega,
 \qquad \omega=\frac{\sqrt{15}}2.
\]
The real part \(-5/2\) describes the size of the mode.  The imaginary part
\(\omega\) is more important for the multiplicity argument: it means that
the direction of the perturbation rotates by angle
\(\omega\log r\).  Consequently, when the inner scale is changed
multiplicatively, the sign seen by a fixed outer detector alternates again
and again.

We must still check that the detector does not lose this signal on the way
to the reflection section.  For the sine-cone model, if \(\Psi_o\) is the complex Frobenius solution
with leading term \(t^{-5/2+i\omega}\), its Wronskian with its conjugate is
\(C\sin^{-6}t\) with \(C\neq0\).  Hence
\[
 \Psi_o(T_0)\neq0,\qquad T_0:=\pi/2,
\]
Thus evaluation at the sine-cone section is a nonzero linear functional on
the odd response space.  Equivalently, the complex transfer coefficient from
the odd Frobenius amplitude to the first closing error is nonzero.  A
particular real phase may of course give zero $F_1$-response; as the inner
scale varies, precisely these phase crossings produce the repeated sign
changes used below.  The imaginary part of the indicial roots is the
logarithmic angular frequency of this rotation.

\subsection{The nonoscillating response}

The second closing condition needs a different parameter.  The even
direction preserves \(f_1=f_2\) and instead changes the relative size of the
third summand.  Unlike the odd direction, its relevant regular mode has a
real exponent, so there is no phase rotation.  On this symmetric subsystem put
\[
 y=\log\frac{f_3}{\sqrt2f_1}.
\]
The sine-cone linearized model equation is
\[
 y''+6\cot t\,y'-5\csc^2t\,y=0,
\]
with regular exponent
\[
 \beta_+=\frac{-5+3\sqrt5}{2}>0.
\]
Writing the equation as
\[
 (\sin^6t\,y')'=5\sin^4t\,y
\]
shows that the regular solution normalized by a positive leading coefficient
stays positive: before any hypothetical first zero the right-hand side is
positive, so $\sin^6t\,y'$ and hence $y'$ remain positive.  Therefore
\(y'(T_0)>0\), where \(T_0=\pi/2\) is the sine-cone reflection time.  At
\(H=0\) in
the symmetric subsystem,
\[
 L_3=\frac23y',
\]
so the even mode is detected nondegenerately by $F_2$.

\begin{proposition}[Triangular detector response]\label{prop:triangular-response}
At the $p_2$ sine cone, the differential of the reflection detector respects
the Weyl splitting.  The odd response contributes at first order only to
$F_1$, and its complex transfer coefficient is nonzero; equivalently, its
real $F_1$-response is a nontrivial sinusoid of the odd phase.  The regular
even mode contributes at first order only to $F_2$, and its transfer
coefficient is positive.  Thus, in parity-adapted coordinates, the response
has nonzero odd transfer and a positive even diagonal entry.
\end{proposition}

\begin{proof}
Weyl parity forces the cross terms to vanish.  The Wronskian calculation in
the preceding subsection gives the nonzero complex odd transfer coefficient;
the integral identity above gives positivity of the even entry.
\end{proof}

Thus the detector is triangular to first order.  The even direction gives a
genuine transverse motion of \(F_2\), while the od direction gives a
nonzero oscillatory motion of \(F_1\).  Schematically,
\[
 \begin{array}{ccl}
 \text{even perturbation} &\longrightarrow& \text{move \(F_2\) through zero},\\
 \text{odd perturbation}  &\longrightarrow& \text{rotate the sign of \(F_1\)}.
 \end{array}
\]
This is the linear core of the detection argument.

\section{Why the oscillation is genuinely present}

The eigenvalue calculation alone does not exclude zero oscillatory
coefficient for the particular threshold trajectory selected in Section~3.
Such vanishing would remove the rotating signal needed for repeated closing.
Weyl symmetry rules out this possibility: the odd variables
satisfy an exactly homogeneous linear equation along the threshold.  Their
asymptotic transport is therefore invertible.  A nonzero odd displacement
cannot be transported into zero asymptotic amplitude.  Since the focal
trajectory is not Weyl symmetric at its collapsing end, its odd displacement
is nonzero, and so is its limiting oscillatory coefficient.

By Proposition~\ref{prop:threshold}, the threshold trajectory converges to
\(p_2\).  Hyperbolicity gives exponential convergence in projective time.
After removing the universal cone decay and rotation from the odd
two-vector, only an integrable perturbation of a constant-coefficient system
remains.  Hence the renormalized odd vector has a finite limiting amplitude.
In radial notation this is the complex coefficient \(Q\) in
\begin{equation}\label{eq:Q}
 z_*(r)=r^{-5/2}\Re(Qr^{i\omega})+o(r^{-5/2}).
\end{equation}

\begin{lemma}\label{lem:Q}
The coefficient $Q$ in \eqref{eq:Q} is nonzero.
\end{lemma}

\begin{proof}
The projective shape variables split under the Weyl involution into even and
odd pairs
\[
 e=(\omega_1+\omega_2,d_1+d_2),\qquad
 o=(\omega_1-\omega_2,d_1-d_2).
\]
Equivariance gives \(F_o(e,-o)=-F_o(e,o)\).  Since \(F_o(e,0)=0\),
Hadamard's formula yields the exact homogeneous equation
\begin{equation}\label{eq:odd-homogeneous-threshold}
 o'=A(e,o)o
\end{equation}
for a smooth \(2\times2\) matrix \(A\).  At \(p_2\),
\[
 A_0=
 \begin{pmatrix}
 0&1\\ -5/18&-5/6
 \end{pmatrix},
\qquad
 \operatorname{spec}A_0=
 -\frac5{12}\pm i\frac{\sqrt{15}}{12}.
\]
Multiplying projective exponents by \(6\) gives the radial pair
\(-5/2\pm i\sqrt{15}/2\).

By Proposition~\ref{prop:threshold}, the threshold converges to \(p_2\).
For completeness, the four constrained shape exponents can be read off
without any sign convention ambiguity in the common-scale variable.  The even
model equation has indicial polynomial
\[
 \beta(\beta-1)+6\beta-5=\beta^2+5\beta-5,
\]
whereas the odd model equation has indicial polynomial
\[
 \alpha(\alpha-1)+6\alpha+10=\alpha^2+5\alpha+10.
\]
Hence the four radial shape exponents are
\[
 \beta_+=\frac{-5+3\sqrt5}{2},\qquad
 \beta_-=\frac{-5-3\sqrt5}{2},\qquad
 -\frac52\pm i\frac{\sqrt{15}}2.
\]
These are exactly the four eigenvalues of the gauge-fixed shape
linearization.  The Hamiltonian constraint removes/reconstructs the common
radial scale rather than contributing another free shape mode; its nonzero
common-scale derivative therefore does not alter these four roots.
On the cone \(H=6/r\), hence
\(\eta=6\log r+\mathrm{const}\); the projective exponents are therefore the
displayed radial exponents divided by \(6\).  Thus \(p_2\) is hyperbolic, with exactly one unstable real direction and a
three-dimensional stable space.  A trajectory converging to \(p_2\) is
therefore eventually contained in its local stable manifold and converges
exponentially.  Since \(A(e,o)\) is smooth,
\[
 B(s):=A(e(s),o(s))-A_0\in L^1([s_0,\infty)).
\]
After a fixed real linear change of variables,
\[
 A_0=-\frac5{12}I+\frac{\sqrt{15}}{12}J,\qquad J^2=-I.
\]
For \(y=To\), set
\[
 w(s)=e^{5s/12}
 \exp\!\left(-\frac{\sqrt{15}}{12}Js\right)y(s).
\]
Then \(w'=C(s)w\) with \(C\in L^1\).  Its fundamental matrix
\(\Phi(s,s_0)\) converges to a matrix \(\Phi_\infty\), and
\[
 \det\Phi_\infty
 =
 \exp\!\left(\int_{s_0}^{\infty}\operatorname{tr}C(\sigma)\,d\sigma\right)
 \neq0.
\]
Thus the renormalized limiting odd amplitude vanishes if and only if
\(o(s_0)=0\).

The radial coefficient \(Q\) differs from this limiting amplitude by a
nonzero complex constant, because \(\log r=s/6+c_0+o(1)\) with exponentially
convergent correction.  Therefore \(Q=0\) would imply \(o(s_0)=0\).  The Weyl-fixed set \(o=0\)
is invariant, and uniqueness for the regular ODE propagates \(o=0\) both
forward and backward on every compact subinterval of positive physical time.
Hence the threshold would satisfy \(f_1=f_2\) for every \(t>0\).  Passing to
the smooth focal limit \(t\downarrow0\) is impossible: \(f_1(0)=0\) whereas
\(f_2(0)>0\).  Hence \(Q\neq0\).
\end{proof}

Thus the threshold does not merely approach a cone that \emph{allows}
oscillation: it approaches that cone with a definite nonzero oscillatory
amplitude.  This is the point that turns the complex indicial roots into an
actual sequence of sign changes.  The remaining task is quantitative.  We
must transport this small rotating signal from the Ricci-flat inner model to
the exact positive-Einstein problem without losing it is the error terms.

\section{Transporting the responses to the Einstein problem}

The Ricci-flat analysis provides a tunable nonoscillating direction and a
nonzero oscillating direction.  Both must persist when the small
positive-Einstein term is restored.

There is an important asymmetry between them.  The even response is of
ordinary size, so a small \emph{absolute} perturbation estimate is enough.
The odd response is much smaller: at the matching scale its size is itself
tending to zero.  An estimate saying only that the Einstein solution is
close to the Ricci-flat one could therefore be larger than the signal we are
trying to detect and could destroy its sign.

What we need in the odd channel is \emph{relative} control: the error must be
small compared with the oscillatory response itself.  Weyl symmetry gives
exactly this.  It makes the odd equation homogeneous, so restoring the
Einstein term may change the coefficients of the odd equation but cannot
create a new odd forcing independent of the existing signal.  This is the
structural reason the tiny phase information survives.

The section therefore has three tasks:
\[
 \begin{array}{ll}
 \text{(i)} & \text{choose the even detuning parameter at the correct scale},\\
 \text{(ii)}& \text{prove ordinary absolute shadowing of the full trajectory},\\
 \text{(iii)}&\text{upgrade the odd part to relative shadowing using symmetry}.
 \end{array}
\]
After these steps the inner Ricci-flat responses can be matched to the outer
positive-Einstein problem.

The projective shape system has only one unstable exponent, namely $\beta_+$.
Choose a small hyperbolic chart about $p_2$.  Fix a sufficiently late regular
time $s_0$ for the threshold trajectory, still inside this chart.  By
continuous dependence, all focal trajectories with $c$ close to $c_*$ are
also in the chart at time $s_0$.  Let $a_{\rm loc}$ be an analytic signed
unstable coordinate there, chosen so that the local stable manifold is
$a_{\rm loc}=0$ and the positive sign is the exit direction toward $S>1$.
Define
\[
 a(c)=a_{\rm loc}(\gamma_c(s_0)).
\]
This fixed-time definition makes $a(c)$ an analytic function of the focal
parameter and avoids any dependence on a diverging entrance time.

\begin{lemma}[Signed threshold coordinate]\label{lem:signed-a}
After orienting $a$, the function $a(c)$ changes sign at $c_*$.  More
precisely, for some finite odd integer $m\ge1$,
\begin{equation}\label{eq:a-expansion}
 a(c)=\kappa(c-c_*)^m+O((c-c_*)^{m+1}),
 \qquad \kappa\ne0.
\end{equation}
Consequently $a$ is a strictly monotone local coordinate and its inverse is
continuous, though it need not be differentiable when $m>1$.
\end{lemma}

\begin{proof}
Proposition~\ref{prop:threshold} gives $a(c_*)=0$.  The positive unstable
eigenvector of the shape Jacobian has equal $\omega$ components and, after
normalization, may be chosen with
\[
 \delta\omega_1=\delta\omega_2>0,
 \qquad \delta d_1=\delta d_2>0.
\]
Thus its positive face points from $S<1$ toward the entrance side $S>1$.
Trajectories in $I_{\mathrm{ent}}$ that approach $c_*$ spend arbitrarily
long time in the $p_2$ chart.  Hyperbolicity and the one-dimensional unstable
shape direction imply that their first exit from a sufficiently small chart
is through the positive unstable face.  Choose the chart with fixed unstable exit faces $a_{\rm loc}=\pm\varepsilon$.
The stable coordinates at the exit tend uniformly to zero as $c\to c_*$.
Since the positive unstable eigenvector has $\delta S>0$, after shrinking the
chart every sufficiently near focal parameter with $a(c)>0$ exits through
$a_{\rm loc}=+\varepsilon$ and crosses from $S<1$ to $S>1$ transversely.  The already stated Proposition~3.4 input from~\cite{Chi} keeps the
trajectory in $S_3$ until that crossing, and the contact derivative has the same sign as
\[
 \Xi=Z_1(X_1-X_3)+Z_2(X_2-X_3).
\]
Hence $\Xi>0$ at the first crossing, so the state lies in $\widehat S_3$.
The already stated Lemma~3.6 input from~\cite{Chi} gives forward trapping
once the crossing has occurred.  Therefore every
sufficiently near parameter with $a(c)>0$ belongs to $I_{\mathrm{ent}}$.
By maximality of $c_*$, parameters immediately beyond the endpoint cannot
have $a(c)>0$.

The local hyperbolic coordinates are real analytic.  The focal curve is also
real analytic in \(c\) on the fixed Eschenburg--Wang branch: the formal
coefficients in the singular initial-value construction depend analytically
on \(c\), and the remainder is the fixed point of a uniformly contracting
operator \(\Phi_c\) which is analytic in \((c,u)\).  Since
\(I-D_u\Phi_c\) is invertible, the analytic implicit-function theorem gives
analytic dependence of the fixed point on \(c\).  Hence \(a(c)\) is real
analytic.  It is not identically zero, because parameters on
the entrance side have $a(c)>0$.  Therefore it has a finite first nonzero
Taylor term at $c_*$.  Since it is positive on one side and nonpositive on the
other, that first nonzero order is odd, which gives \eqref{eq:a-expansion}.
The derivative of the leading odd monomial has fixed sign off $c_*$, so after
shrinking the interval $a$ is strictly monotone.
\end{proof}

The scale in the next definition is chosen so that the unstable even mode,
whose cone behaviour is $r^{\beta_+}$, reaches order one after transport from
the focal scale $b$ to a fixed outer scale.  We therefore parameterize the
inner family directly by its unstable amplitude and set
\begin{equation}\label{eq:detune}
 a=\mu b^{\beta_+},\qquad |\mu|\le\mu_0.
\end{equation}
Lemma~\ref{lem:signed-a} determines a unique nearby focal parameter
$c=c(b,\mu)$ on the chosen branch.  It depends continuously on $(b,\mu)$ and
tends to $c_*$ uniformly as $b\downarrow0$.  We record the uniform detuning
modulus
\begin{equation}\label{eq:rho-detune}
 \rho(b)=\sup_{|\mu|\le\mu_0}|c(b,\mu)-c_*|\longrightarrow0.
\end{equation}
In fact the finite-order expansion \eqref{eq:a-expansion} gives the uniform
rate
\begin{equation}\label{eq:rho-power-rate}
 \rho(b)=O\!\left(b^{\beta_+/m}\right).
\end{equation}
Indeed, locally $|c-c_*|\le C|a|^{1/m}$ and
$|a|\le\mu_0b^{\beta_+}$.  We will only use continuity of the inverse
$c=c(a)$, together with the power bound \eqref{eq:rho-power-rate} when a
logarithmic loss occurs in the outer matching.  The parameter \(\mu\) is the target coefficient of the regular even mode.
The corresponding exact outer coefficient will be introduced only when the
outer matching begins.

Blow up the positive-Einstein solution at the focal scale by
\[
 t=b\tau,\qquad f_i=bF_i.
\]
Equivalently,
\[
 b^{-2}g=d\tau^2+\sum_i F_i(\tau)^2Q|_{\mathfrak m_i}.
\]
Thus \(\tau\) is again unit-speed for the rescaled metric \(b^{-2}g\).
This is a genuine homothetic blow-up used to obtain the Ricci-flat inner
limit: its Einstein constant is \(6b^2\), which tends to zero.  It is
not a normalization of the orbit-space length of the original
\(\Ric=6g\) metric.  Write \(s=\log\tau\) only as an evolution parameter for the normalized
dynamical system; \(s\) is not the arclength coordinate of the metric.  In the
projective gauge-fixed shape chart (the
radial $\alpha=1$ mode being reconstructed from the constraint),
\begin{equation}\label{eq:inner}
 U'=F_{\RF}(U)+b^2e^{2s}R(U).
\end{equation}
We match the inner and outer descriptions at an intermediate scale
$b\ll\delta_b\ll1$.  The convenient choice
\[
 \delta_b=b^{1/2},\qquad
 s_b=\log(\delta_b/b)=\frac12\log(1/b).
\]
The even unstable component at the overlap is
$\mu\delta_b^{\beta_+}=o(1)$.

\subsection{Why the small oscillation cannot be additively forced}

Suppose the odd Ricci-flat signal has size \(A\ll1\).  An absolute
positive-Einstein perturbation of size \(o(1)\) does not preserve its sign
unless the error is \(o(A)\).  Weyl symmetry supplies the required relative
control.
The Weyl involution fixes the even background and reverses the odd variable.
Consequently every term in the odd equation contains the odd variable
itself.  The Einstein perturbation can amplify, damp, or rotate an existing
odd signal, but it cannot manufacture an independente one.  We now record
this exact factorization before making estimates.

For \(d=2\), the orbit-Ricci convention \eqref{eq:intro-orbit-ricci}
has cyclic differences
\begin{align}
 r_1-r_2&=
 \frac{(f_1^2-f_2^2)(f_1^2+f_2^2-3f_3^2)}
 {2f_1^2f_2^2f_3^2},\label{eq:ricci-diff12}\\
 r_2-r_3&=
 \frac{(f_2^2-f_3^2)(f_2^2+f_3^2-3f_1^2)}
 {2f_1^2f_2^2f_3^2},\\
 r_3-r_1&=
 \frac{(f_3^2-f_1^2)(f_3^2+f_1^2-3f_2^2)}
 {2f_1^2f_2^2f_3^2}.
\end{align}
Thus all curvature-difference identities in the positive-Einstein part use
one convention.  Since in our normalization \(L_i'=r_i-6-HL_i\), the Einstein constant
cancels from the odd difference equation.  Put $m^2=f_1f_2$.  Then
\begin{equation}\label{eq:oddexact}
 z''+Hz'
 =
 \frac{\sinh z\,[2m^2\cosh z-3f_3^2]}
 {m^2f_3^2}.
\end{equation}
At the sine-cone background,
\[
 m^2=\frac{\sin^2t}{5},\qquad
 f_3^2=\frac{2\sin^2t}{5},\qquad H=6\cot t,
\]
so linearizing \eqref{eq:oddexact} gives
\[
 z''+6\cot t\,z'+10\csc^2t\,z=0,
\]
in agreement with \eqref{eq:odd-sine}.  This also checks the coefficient and
sign conventions independently.

The statement ``odd in $z$'' is a consequence of the Weyl involution, not an
extra assumption on $f_3$.  Indeed the reflection exchanging the first two
Wallach summands acts by
\[
 (f_1,f_2,f_3)\longmapsto(f_2,f_1,f_3),\qquad
 z\longmapsto-z,
\]
while $m^2=f_1f_2$, $f_3$ and $H$ are fixed.  Hence the coefficient
\[
 \frac{2m^2\cosh z-3f_3^2}{m^2f_3^2}
\]
is Weyl-even and the factor $\sinh z$ is Weyl-odd.  Thus
\eqref{eq:oddexact} is equivariant and the subsystem $z=z'=0$ is invariant.
Equivalently, after writing the odd equation as a first-order system, every
term contains the odd two-vector.  Positive Einstein may change its
coefficients through the Weyl-even background, but it cannot create an
additive odd forcing.

\subsection{Absolute control versus relative odd control}

We can now state precisely the two levels of approximation.  For the even
variables and the background geometry, ordinary absolute closeness is enough.
For the oscillating variable, the relevant quantity is instead the ratio
\[
 \frac{\text{Einstein odd signal}-\text{Ricci-flat odd signal}}
      {\text{Ricci-flat odd signal}}.
\]
The goal is to prove that this ratio tends to zero.  This distinction is the
reason for separating the next two estimates.

Let $U_{\PE}$ and $U_{\RF}$ have the same normalized focal data.  The
four-dimensional gauge-fixed shape linearization at $p_2$ is hyperbolic, with
one unstable exponent $\beta_+<2$ and three stable exponents.  We record the
shadowing estimate explicitly.

\begin{lemma}[Absolute inner shadowing]\label{lem:absolute-inner}
For $s_0$ fixed sufficiently far inside the cone chart,
\begin{equation}\label{eq:pointwise}
 \|U_{\PE}(s)-U_{\RF}(s)\|\le Cb^2e^{2s},
 \qquad s_0\le s\le s_b,
\end{equation}
uniformly for $|\mu|\le\mu_0$.
\end{lemma}

\begin{proof}
Put $W=U_{\PE}-U_{\RF}$.  On a fixed small cone chart,
\[
 W'=A(s)W+N(s,W)+b^2e^{2s}R(s),
 \qquad |N(s,W)|\le\varepsilon_0|W|,
\]
where $A(s)$ has a uniform exponential dichotomy obtained by perturbing the
constant linearization at $p_2$.  On each stable block the forward kernel is
bounded by $Ce^{-\gamma(s-\sigma)}$ for some $\gamma>0$, whereas on the
single unstable block it is bounded by
$Ce^{\beta_+(s-\sigma)}$.  Consequently the forcing satisfies
\[
 \int_{s_0}^s e^{-\gamma(s-\sigma)}b^2e^{2\sigma}\,d\sigma
 \le Cb^2e^{2s},
 \qquad
 \int_{s_0}^s e^{\beta_+(s-\sigma)}b^2e^{2\sigma}\,d\sigma
 \le \frac{b^2e^{2s}}{2-\beta_+}.
\]
Use the weighted norm
\[
 \|W\|_*=
 \sup_{s_0\le s\le s_b} b^{-2}e^{-2s}|W(s)|.
\]
On the fixed focal-to-$s_0$ segment, smooth dependence on the Einstein
constant gives $|W(s_0)|\le Cb^2$.  Variation of constants, including this
initial term, therefore gives
\[
 \|W\|_*\le C+C\varepsilon_0\|W\|_*.
\]
Choosing the cone chart so that $C\varepsilon_0<1/2$ yields a uniform bound
for $\|W\|_*$ and proves \eqref{eq:pointwise}.  Notice that this weighted
argument has no factor growing with the length $s_b-s_0$.
\end{proof}

The preceding lemma says that the full shape error is
$O(b^2e^{2s})$.  That is sufficient for the even/background variables, but
not for the odd phase: the odd solution itself decays like $e^{-5s/2}$.
We therefore compare the two odd solutions after removing their common decay
and rotation.  Introduce
\[
 v=\frac{z_s+\frac52z}{\omega},\qquad
 V=\binom zv,\qquad
 J=\begin{pmatrix}0&1\\-1&0\end{pmatrix}.
\]
At the cone the odd system is
\[
 V'=\left(-\frac52I+\omega J\right)V.
\]
Along the threshold it has the form
\begin{equation}\label{eq:odd-rotating-system}
 V'=\left(-\frac52I+\omega J+B_*(s)\right)V,
 \qquad \int_{s_0}^{\infty}\|B_*(s)\|\,ds<\infty.
\end{equation}
Thus
\[
 \widetilde V=e^{5(s-s_0)/2}e^{-\omega J(s-s_0)}V
\]
satisfies an $L^1$-coefficient equation.  Its fundamental matrix and inverse
are uniformly bounded.  The same remains true for the detuned Ricci-flat
family on $[s_0,s_b]$: the even unstable coordinate has size
$O(|\mu|b^{\beta_+}e^{\beta_+s})$, whose integral is
$O(|\mu|\delta_b^{\beta_+})$, while the stable-coordinate change at the
fixed entry section is $O(\rho(b))$ and decays exponentially.  Thus the total
$L^1$ coefficient change is $O(\rho(b)+\delta_b^{\beta_+})$.  Hence there are constants $0<c<C$, independent of $b$ and
$|\mu|\le\mu_0$, such that
\begin{equation}\label{eq:rotating-two-sided}
 c e^{-5(s-\sigma)/2}|V(\sigma)|
 \le |V(s)|
 \le C e^{-5(s-\sigma)/2}|V(\sigma)|,
 \qquad s_0\le\sigma\le s\le s_b.
\end{equation}

\begin{lemma}[Parity-weighted inner comparison]\label{lem:weighted-inner}
At the overlap $t=\delta_b=b^{1/2}$,
\begin{equation}\label{eq:weighted-inner}
 |V_{\PE}(s_b)-V_{\RF}(s_b)|
 \le C\delta_b^2\,|V_{\RF}(s_b)|,
\end{equation}
uniformly for $|\mu|\le\mu_0$.  Equivalently, the
positive-Einstein perturbation changes the rotating odd amplitude only by the
relative factor $1+O(\delta_b^2)$.
\end{lemma}

\begin{proof}
The exact factorization \eqref{eq:oddexact} implies that the odd first-order
system is homogeneous.  In parity-adapted variables it has the form
\begin{equation}\label{eq:odd-homogeneous}
 V_o'=A_o(U_e,V_o)V_o
\end{equation}
for the Ricci-flat equation.  The positive-Einstein equation has the same
Weyl parity and therefore can be written
\begin{equation}\label{eq:odd-homogeneous-PE}
 V_o'=\bigl(A_o(U_e,V_o)+b^2e^{2s}B_o(U_e,V_o;b)\bigr)V_o.
\end{equation}
Here $U_e$ denotes the Weyl-even variables.  Since the odd vector field is
odd under $V_o\mapsto -V_o$, Hadamard's formula gives a smooth factorization
for which the coefficient matrices are even:
\[
 A_o(U_e,-V_o)=A_o(U_e,V_o),\qquad
 B_o(U_e,-V_o;b)=B_o(U_e,V_o;b).
\]
Consequently
\begin{equation}\label{eq:Ao-odd-derivative}
 \|D_{V_o}A_o(U_e,V_o)\|+\|D_{V_o}B_o(U_e,V_o;b)\|
 \le C|V_o|
\end{equation}
in a fixed cone chart.

Choose $s_0$ fixed and sufficiently far inside that chart.  On the fixed
focal-to-$s_0$ segment, smooth dependence on the Einstein perturbation parameter \(b^2\)
gives
\[
 |V_{\PE}(s_0)-V_{\RF}(s_0)|\le Cb^2.
\]
Because $Q\ne0$, the threshold rotating amplitude at $s_0$ is nonzero.
After shrinking the detuning range, the same is true uniformly for the
detuned Ricci-flat family, and hence
\begin{equation}\label{eq:weighted-initial}
 |V_{\PE}(s_0)-V_{\RF}(s_0)|
 \le Cb^2|V_{\RF}(s_0)|.
\end{equation}

Let $A_{\RF}(s)=A_o(U_{e,\RF}(s),V_{\RF}(s))$ and let
$\Phi_{\RF}(s,\sigma)$ be its fundamental matrix.  By
\eqref{eq:rotating-two-sided}, after removing the universal factor
$e^{-5(s-\sigma)/2}$ the matrices $\Phi_{\RF}$ and
$\Phi_{\RF}^{-1}$ are uniformly bounded on $[s_0,s_b]$.
Write
\[
 V_{\PE}(s)=\Phi_{\RF}(s,s_0)Y(s),\qquad
 V_{\RF}(s)=\Phi_{\RF}(s,s_0)Y_0,
\]
where $Y_0=V_{\RF}(s_0)$, and set $Z=Y-Y_0$.
Then
\begin{equation}\label{eq:Y-difference}
 Z'=\Phi_{\RF}(s_0,s)\,\Delta A(s)\,
     \Phi_{\RF}(s,s_0)(Y_0+Z),
\end{equation}
where $\Delta A$ is the difference between the positive-Einstein odd
coefficient matrix and $A_{\RF}$.

The absolute comparison \eqref{eq:pointwise} gives
\[
 |U_{e,\PE}(s)-U_{e,\RF}(s)|\le Cb^2e^{2s}.
\]
Using smoothness, \eqref{eq:Ao-odd-derivative}, and the bootstrap
$|Z|\le \frac12|Y_0|$, we therefore obtain
\begin{equation}\label{eq:DeltaA-bound}
 \|\Delta A(s)\|
 \le Cb^2e^{2s}
     +C|V_{\RF}(s)|^2\frac{|Z(s)|}{|Y_0|}.
\end{equation}
Indeed the first term contains both the change of the even background and
the explicit positive-Einstein term in \eqref{eq:odd-homogeneous-PE}; the
second term is the change of the even coefficient matrix caused by replacing
$V_{\RF}$ by $V_{\PE}$.

The two-sided estimate \eqref{eq:rotating-two-sided} implies
\[
 \int_{s_0}^{s_b}|V_{\RF}(s)|^2\,ds\le C,
\]
uniformly in $b$ and $|\mu|\le\mu_0$.  Dividing
\eqref{eq:Y-difference} by $|Y_0|$ and using
\eqref{eq:DeltaA-bound} gives
\[
 r(s)
 \le Cb^2
   +C\int_{s_0}^s b^2e^{2\sigma}\,d\sigma
   +C\int_{s_0}^s |V_{\RF}(\sigma)|^2r(\sigma)\,d\sigma,
 \qquad r=\frac{|Z|}{|Y_0|}.
\]
Gronwall now involves only an integrable coefficient.  Since
\[
 \int_{s_0}^{s_b}b^2e^{2\sigma}\,d\sigma\le C\delta_b^2,
\]
we obtain
\[
 r(s)\le C\delta_b^2,
 \qquad s_0\le s\le s_b.
\]
For $b$ small this closes the bootstrap.  Returning through
$\Phi_{\RF}$ and using \eqref{eq:rotating-two-sided} yields
\[
 |V_{\PE}(s)-V_{\RF}(s)|
 \le C\delta_b^2|V_{\RF}(s)|,
 \qquad s_0\le s\le s_b,
\]
and in particular \eqref{eq:weighted-inner} at $s=s_b$.
\end{proof}

For later use set
\[
 \tau_b:=\frac{\delta_b}{b}=b^{-1/2};
\]
this is the Ricci-flat radial size of the overlap point.

Changing the Ricci-flat focal parameter from $c_*$ to $c(b,\mu)$ changes
the stable coordinates at the fixed entry section by $O(\rho(b))$ and the
unstable shape coordinate by exactly $\mu b^{\beta_+}$.  Stable contributions
are integrable in the rotating frame, while the growing unstable contribution
has integral $O(\delta_b^{\beta_+})$.  Lemma~\ref{lem:weighted-inner} gives
the relative positive-Einstein error.  Hence, for some $\eta_*>0$, the complex
rotating odd amplitude at the overlap differs from the threshold amplitude by
\begin{equation}\label{eq:overlapodd}
 O\!\left(\rho(b)+\tau_b^{-\eta_*}
          +\delta_b^{\beta_+}+\delta_b^2\right)
\end{equation}
relatively, uniformly for bounded $\mu$.

\section{Uniform outer matching}

At the overlap scale the positive-Einstein solution carries the same odd
phase as the Ricci-flat model up to a vanishing relative error, while the
even detuning has the prescribed leading size.  These two pieces are then
propagated through a fixed outer corridor to the parameter-dependent first
\(H=0\) section.

The outer problem is regular.  Its role is therefore not to create the
oscillation, but to \emph{read} it.  The even mode is converted into the
second detector component \(F_2\); the odd mode is converted into \(F_1\).
Uniformity in the bounded detuning parameter is essential, because the final
existence argument uses whole parameter rectangles rather than a single
trajectory.

At the overlap the inner analysis has produced an even coefficient close to
$\mu$ and a nonzero odd coefficient carrying the phase
$\omega\log(1/b)$.  We now transport both to the first $H=0$ section.  We
keep the overlap
\[
 \delta_b=b^{1/2},\qquad \tau_b=b^{-1/2}.
\]
Hyperbolicity of $p_2$ sharpens \eqref{eq:Q}: for some spectral-gap number
$\eta>0$,
\begin{equation}\label{eq:Qrate}
 z_*(\tau)
 =\Re\!\left(Q\tau^{-5/2+i\omega}\right)
 +O(\tau^{-5/2-\eta}).
\end{equation}
The same estimate, together with the rotating-frame bounds of the preceding
section, holds uniformly for the detuned Ricci-flat and positive-Einstein
families.

The even regular coefficient requires a separate normalization.  Let
\(\xi_e\) be the unstable even coordinate in the \(p_2\) chart, normalized so
that the corresponding outer Frobenius mode has leading behavior
\(t^{\beta_+}\) as \(t\downarrow0\).  The inner
normal form and \eqref{eq:detune} give
\[
 \xi_{e,\RF}(s_b)
 =\mu\delta_b^{\beta_+}
  \left(1+O(\tau_b^{-\eta_*})+O(\delta_b^{\beta_+})\right).
\]
The positive-Einstein comparison \eqref{eq:pointwise} changes this by
$O(\delta_b^2)$.  Hence there is a regular outer coefficient
$\widehat\mu=\widehat\mu(b,\mu)$ such that
\begin{equation}\label{eq:muhat}
 \widehat\mu
 =\mu+O(\kappa_b),\qquad
 \kappa_b=
 \rho(b)+\tau_b^{-\eta_*}+\delta_b^{\beta_+}
 +\delta_b^{\,2-\beta_+}.
\end{equation}
The exponent $2-\beta_+$ is positive.  Thus the exact even overlap state is
matched to the symmetric regular-singular family with coefficient
$\widehat\mu$, and $\widehat\mu\to\mu$ uniformly for bounded $\mu$.

Consequently the complex odd amplitude at the overlap can be written
\begin{equation}\label{eq:inner-overlap-complex}
 q_b
 =Q\tau_b^{-5/2+i\omega}
 \left(1+O(\tau_b^{-\eta})+O(\kappa_b)\right),
\end{equation}
uniformly for $|\mu|\le\mu_0$.  Here and below a complex amplitude denotes the
coefficient whose real part gives the real odd solution.

Fix once and for all a small regular time $t_*>0$.  For each sufficiently
small parameter $\nu$, regular-singular theory at the $p_2$ sine cone gives a
symmetric outer solution $G_\nu$ with
\[
 y_\nu(t)
 =\nu t^{\beta_+}\bigl(1+O(t^\sigma)\bigr)
  +O(\nu^2t^{2\beta_+})
 \qquad (0<t\le t_*),
\]
for some $\sigma>0$, uniformly for bounded $\nu$.
Starting from $t_*$, ordinary regular ODE dependence propagates these
solutions through a fixed compact positive corridor containing
\(t=T_0=\pi/2\) for the sine cone.  After shrinking \(\mu_0\), every
\(|\nu|\le2\mu_0\) has its own unique first \(H=0\) event \(T_\nu\) in this
corridor: existence follows from closeness to the sine cone, and uniqueness
and uniform transversality follow from
\eqref{eq:Hmono}.  Define the even detector on this moving section by
\[
 \mathcal E(\nu)=L_3(G_\nu(T_\nu)).
\]
Here \(T_0=\pi/2\) only for the sine cone.  Uniform transversality and smooth
dependence give
\[
 T_\nu=T_0+O(\nu),
\]
as well as smoothness of \(\mathcal E\).  At $\nu=0$
the background has $y\equiv0$, so the first variation of the moving event
does not contribute to $y'$.  Since on the symmetric $H=0$ section
$L_3=\frac23y'$, the even calculation of Section~4.2 gives
\begin{equation}\label{eq:E}
 \mathcal E(\nu)=C_e\nu+O(\nu^2),\qquad C_e>0.
\end{equation}
In the matching argument below the symmetric background is
$G_{\widehat\mu}$.  By \eqref{eq:muhat} and smoothness,
\begin{equation}\label{eq:E-muhat}
 \mathcal E(\widehat\mu)=\mathcal E(\mu)+O(\kappa_b).
\end{equation}

We next normalize the odd outer solution.  Let $\Psi_{o,\nu}$ be the complex
solution of the linear odd equation along $G_\nu$ satisfying
\begin{equation}\label{eq:outer-Frobenius}
 \Psi_{o,\nu}(t)
 =t^{-5/2+i\omega}\bigl(1+O(t^{\sigma})\bigr)
 \qquad(t\downarrow0)
\end{equation}
for some $\sigma>0$, uniformly for $|\nu|\le2\mu_0$.  Parameter-dependent
regular-singular theory makes $\Psi_{o,\nu}$ continuous in $\nu$ up to
$t_*$, and ordinary linear ODE dependence propagates it continuously through
the fixed regular corridor.  Let $\mathcal D_\nu$ denote the linearized
$F_1$ detector at the transverse $H=0$ event of $G_\nu$, and set
\begin{equation}\label{eq:Co-def}
 C_o(\nu)=\mathcal D_\nu(\Psi_{o,\nu}).
\end{equation}
Then $C_o$ is $C^1$ (indeed real analytic) on the fixed compact detuning interval, hence uniformly Lipschitz.  In particular,
\begin{equation}\label{eq:Co-muhat}
 C_o(\widehat\mu)=C_o(\mu)+O(\kappa_b).
\end{equation}
At $\nu=0$ the background is the sine cone and
Section~4.1 shows \(\Psi_o(T_0)\ne0\) at the sine-cone event
\(T_0=\pi/2\).  An odd first variation does not move
the $H=0$ event to first order because $H$ is Weyl-even.  Since
$F_1=f_1-f_2=f_0z+O(z^2)$ there with $f_0>0$, this is precisely
\begin{equation}\label{eq:Co-nonzero}
 C_o(0)\ne0.
\end{equation}

The next estimate is the outer analogue of the parity-weighted inner
comparison.  Its key gain is quadratic dependence of the even background on
the odd variable: a first-order odd perturbation cannot feed directly into
the even equations.

\begin{lemma}[Weighted outer parity control]\label{lem:outer-parity}
Let $E$ denote the Weyl-even weighted outer state of the exact matched
solution and $E_{\widehat\mu}$ the corresponding state of the symmetric outer
solution $G_{\widehat\mu}$.  Put
\[
 A=b^{5/2},\qquad
 \vartheta_b=
 \kappa_b+\delta_b^\sigma+\delta_b^2
 +\left(\frac b{\delta_b}\right)^5.
\]
For $t\in[\delta_b,t_*]$,
\begin{equation}\label{eq:weighted-even-state}
 \|E(t)-E_{\widehat\mu}(t)\|_{\mathrm w}\le C\vartheta_b,
\end{equation}
and, writing the exact odd equation as
\[
 z''+H(E,z)z'=a(E,z)z,
\]
one has
\begin{align}
 |H(E,z)-H(E_{\widehat\mu},0)|
 &\le \frac{C}{t}\bigl(\vartheta_b+|z|^2\bigr),
 \label{eq:H-weighted-parity}\\
 |a(E,z)-a(E_{\widehat\mu},0)|
 &\le \frac{C}{t^2}\bigl(\vartheta_b+|z|^2\bigr).
 \label{eq:a-weighted-parity}
\end{align}
\end{lemma}

\begin{proof}
Use the weighted scale variables
\[
 x_e=\frac{\sqrt{f_1f_2}}{t},\qquad y_e=\frac{f_3}{t},
\]
together with the corresponding weighted even logarithmic derivatives and
the odd first-order vector $V$.  The Einstein equations are regular-singular
in these variables:
\[
 tE'=F_e(E,V),\qquad tV'=F_o(E,V).
\]
Weyl equivariance gives
\[
 F_e(E,-V)=F_e(E,V),\qquad F_o(E,-V)=-F_o(E,V),
\]
and therefore $D_VF_e(E,0)=0$.  The reference solution
$G_{\widehat\mu}$ lies on the regular Frobenius branch.  The exact matched
solution retains a small stable even tail inherited from the inner problem;
at $t=\delta_b$ this tail is $O(\tau_b^{-\eta_*})$ and is already included
in $\kappa_b$.  After the regular coefficient is matched by
\eqref{eq:muhat}, only the complementary weighted even modes remain.  The
uniformity as $\delta_b\downarrow0$ can be proved directly from the
indicial splitting.  The genuine even \emph{shape} variable (distinct from the weighted scale
\(y_e\) above) is
\[
 y=\log\frac{f_3}{\sqrt2\,\sqrt{f_1f_2}}.
\]
At the \(p_2\) sine cone its linearized model equation is
\[
 y''+6\cot t\,y'-5\csc^2t\,y=0.
\]
For the nearby backgrounds \(G_\nu\) the coefficients are perturbed; only
their common regular-singular leading part at \(t=0\) is used to obtain the
same indicial splitting uniformly in \(\nu\).  The model indicial polynomial at $t=0$ is
\[
 \lambda(\lambda-1)+6\lambda-5
 =\lambda^2+5\lambda-5,
\]
with roots
\[
 \beta_+=\frac{-5+3\sqrt5}{2}>0,
 \qquad
 \beta_-=\frac{-5-3\sqrt5}{2}<0.
\]
Thus the gap between the regular and complementary shape modes is the fixed
number
\[
 \beta_+-\beta_-=3\sqrt5.
\]

For the nearby symmetric backgrounds $G_\nu$, $|\nu|\le2\mu_0$, put
$X=(y,ty')^T$.  After decreasing $t_*$ if necessary, the linearized weighted
shape equation has the form
\begin{equation}\label{eq:even-rs-system}
 tX'=\bigl(A_0+t^\sigma A_\nu(t)\bigr)X+F,
 \qquad
 A_0=
 \begin{pmatrix}
 0&1\\ 5&-5
 \end{pmatrix},
\end{equation}
where $A_\nu$ is uniformly bounded in both variables.  The eigenvalues of
$A_0$ are precisely $\beta_+$ and $\beta_-$.  Passing to logarithmic time
\(s=\log t\) as an auxiliary logarithmic independent variable, and
conjugating by the fixed eigenvector matrix gives
\[
 Z'=
 \left[
 \begin{pmatrix}\beta_+&0\\0&\beta_-\end{pmatrix}
 +R_\nu(s)
 \right]Z+\widetilde F,
 \qquad
 \|R_\nu(s)\|\le Ce^{\sigma s}.
\]
Moreover
\[
 \sup_{|\nu|\le2\mu_0}
 \int_{-\infty}^{\log t_*}\|R_\nu(s)\|\,ds
 \le \frac{C}{\sigma}t_*^\sigma.
\]
Choose $t_*$ so that this quantity is smaller than one quarter of the spectral
gap.  The Volterra equations for the two eigendirections are then contractions
in the norms $e^{-\beta_\pm s}|Z_\pm(s)|$.  They give fundamental solutions,
uniformly for $|\nu|\le2\mu_0$,
\[
 u_{+,\nu}(t)=t^{\beta_+}(1+O(t^\sigma)),
 \qquad
 u_{-,\nu}(t)=t^{\beta_-}(1+O(t^\sigma)),
\]
with the same estimates after one weighted derivative.  Their Wronskian
therefore satisfies
\[
 W_\nu(t)
 =(\beta_--\beta_+)
  t^{\beta_++\beta_--1}(1+O(t^\sigma))
 =-3\sqrt5\,t^{-6}(1+O(t^\sigma)).
\]
Thus $t^6W_\nu(t)$ is bounded away from zero uniformly in $\nu$.

The preceding calculation controls the genuine even \emph{shape} mode.  We
also record explicitly why it controls the full reduced even state used below.
Separate the weighted trace variable
\[
 q=tH-6
\]
from the traceless even shape variables.  Put
\[
 tL_i=1+\ell_i.
\]
Then, identically,
\[
 q=2(\ell_1+\ell_2+\ell_3).
\]
Raychaudhuri's equation gives the exact weighted trace identity
\begin{equation}\label{eq:q-exact}
 t q'
 =
 -q-2\sum_{i=1}^3\ell_i^2-6t^2.
\end{equation}
Indeed,
\[
 tq'=tH+t^2H'
 =(q+6)-2\sum_i(1+\ell_i)^2-6t^2,
\]
which is exactly \eqref{eq:q-exact}.  After subtracting the symmetric
reference solution the final term cancels.

The genuine even shape variable is
\[
 y=\log\frac{f_3}{\sqrt2\,\sqrt{f_1f_2}},
\]
so
\[
 ty'=\ell_3-\frac{\ell_1+\ell_2}{2}.
\]
Consequently the even logarithmic derivatives are reconstructed exactly by
\begin{equation}\label{eq:ell-reconstruct}
 \frac{\ell_1+\ell_2}{2}
 =\frac q6-\frac13ty',
 \qquad
 \ell_3=\frac q6+\frac23ty'.
\end{equation}
The remaining difference \(\ell_1-\ell_2\) is Weyl-odd.  This, after
subtracting two nearby solutions, the quadratic term in \eqref{eq:q-exact}
has the structure
\[
 \Delta\!\left(\sum_i\ell_i^2\right)
 =
 O(t^{\sigma_0})\Delta q
 +O(t^{\sigma_0})|X_{\rm sh}|
 +O(|V|^2),
\]
uniformly on the bounded detuning interval, where
\(X_{\rm sh}=(y,ty')\) and one may take
\(\sigma_0=\min\{\beta_+,\sigma,2\}>0\).
Hence
\begin{equation}\label{eq:radial-weighted}
 t\,\Delta q'
 =
 -\Delta q
 +O(t^{\sigma_0})\Delta q
 +O(t^{\sigma_0})|X_{\rm sh}|
 +O(|V|^2).
\end{equation}
The coefficient \(-1\) is therefore the exact radial trace indicial exponent.
The free trace mode is \(t^{-1}\) and is absent on the regular branch; the
forced trace Green kernel is bounded by \(C(\xi/t)\).

It remains only to reconstruct the common weighted orbit scale.  The master
constraint \eqref{eq:master-hamiltonian} is equivalently
\[
 H^2-|L|^2=\operatorname{Scal}(g_t)-30.
\]
As a direct check, on the \(p_2\) sine cone one has
\(H=6\cot t\), \(|L|^2=6\cot^2t\), and
\(\operatorname{Scal}(g_t)=30\csc^2t\), so
\eqref{eq:master-hamiltonian} is identically satisfied.  Since every isotropy
summand has real dimension two, the weighted form is
\begin{equation}\label{eq:weighted-Hamiltonian}
 (tH)^2-2\sum_{i=1}^3(tL_i)^2
 -t^2\operatorname{Scal}(g_t)+30t^2=0.
\end{equation}
Write the weighted orbit factors as \(x_i=f_i/t\).  At the \(p_2\) cone
their ratios are fixed and
\[
 t^2\operatorname{Scal}(g_t)=30.
\]
Under a common logarithmic rescaling \(x_i\mapsto e^r x_i\), the weighted
orbit scalar curvature changes by the factor \(e^{-2r}\).  The remaining
terms in \eqref{eq:weighted-Hamiltonian} are unchanged under this common scale
variation at fixed weighted logarithmic derivatives.  Since the scalar term
enters with a minus sign, the derivative of the left-hand side with respect
to \(r\) at the cone is
\[
 \frac{d}{dr}\bigl(-30e^{-2r}\bigr)\Big|_{r=0}=60\neq0.
\]
The implicit-function theorem consequently reconstructs the common radial
scale as a uniformly smooth function of \(q\), the even shape variables, and
the odd variables in a fixed cone neighborhood.  Thus the full weighted even
state is controlled by the trace mode together with \(X_{\rm sh}\).

At this point the differential-equation conventions are fixed by
\eqref{eq:master-ricci}--\eqref{eq:master-hamiltonian}.  In particular, the
positive-Einstein equations, their Ricci-flat blow-up, the odd difference
equation, the weighted trace equation, and the common-scale reconstruction
all use the same signs.

After the coefficient of the sole regular shape mode \(t^{\beta_+}\) has
been absorbed into \(\widehat\mu\), the complementary full even state has only
the two singular indicial rates
\[
 \beta_-<0,\qquad -1<0.
\]
Both are excluded as free data by regularity, and their forced kernels are
uniform on every shrinking interval.  Combining the shape kernel below with
the trace kernel gives
\[
 \|K_\nu(t,\xi)\|
 \le C\left(\frac{\xi}{t}\right)^\gamma,
 \qquad
 \gamma=\min\{1,-\beta_-\}=1,
\]
with \(C\) independent of the lower endpoint.

The coefficient of the regular mode $u_{+,\nu}$ has already been absorbed
into $\widehat\mu$.  On the complementary shape mode we have, for
$0<\xi\le t\le t_*$,
\[
 \left|
 \frac{u_{-,\nu}(t)}{u_{-,\nu}(\xi)}
 \right|
 \le C\left(\frac{\xi}{t}\right)^{-\beta_-}.
\]
The variation-of-constants formula together with the uniform Wronskian bound
therefore yields the complementary weighted Green kernel estimate
\begin{equation}\label{eq:uniform-even-green}
 \|K_\nu(t,\xi)\|
 \le C\left(\frac{\xi}{t}\right)^\gamma,
 \qquad
 \gamma=\min\{1,-\beta_-\}=1,
 \qquad
 0<\xi\le t\le t_*,
\end{equation}
with $C$ independent of $\nu$ and, crucially, independent of the lower
endpoint.  Hence
\[
 \sup_{0<\delta<t_*}
 \sup_{t\in[\delta,t_*]}
 \int_\delta^t
 \|K_\nu(t,\xi)\|\,\frac{d\xi}{\xi}
 \le \frac C\gamma.
\]
The complementary weighted Green operator is therefore uniformly bounded on
every shrinking interval $[\delta_b,t_*]$.  This proves that the controlled
stable tail at $t=\delta_b$ is propagated with no loss as
$\delta_b\downarrow0$.  Variation of constants gives
\[
 \|E(t)-E_{\widehat\mu}(t)\|_{\mathrm w}
 \le C\left(
 \kappa_b+\delta_b^\sigma+\delta_b^2
 +\int_{\delta_b}^{t}|V(\xi)|^2\,\frac{d\xi}{\xi}\right).
\]
The odd Frobenius bound $|V(\xi)|\le CA\xi^{-5/2}$ yields
\[
 \int_{\delta_b}^{t}|V(\xi)|^2\,\frac{d\xi}{\xi}
 \le CA^2\delta_b^{-5}
 =C\left(\frac b{\delta_b}\right)^5,
\]
which proves \eqref{eq:weighted-even-state}.

The exact factorization \eqref{eq:oddexact}, with
$m=\sqrt{f_1f_2}=tx$ and $f_3=ty$, gives
\[
 a(E,z)
 =\frac1{t^2}\frac{\sinh z}{z}
   \frac{2x^2\cosh z-3y^2}{x^2y^2}.
\]
Thus $t^2a$ is smooth in the weighted even variables and is even in $z$.
Likewise $tH$ is a smooth weighted even variable.  Taylor's theorem,
\eqref{eq:weighted-even-state}, and the vanishing of the first $z$-derivative
at $z=0$ give
\eqref{eq:H-weighted-parity}--\eqref{eq:a-weighted-parity}.
\end{proof}

We can now state the central matching result.  The second detector converges
to the nonoscillating even response $\mathcal E(\mu)$.  After removing its
natural size $b^{5/2}$, the first detector converges to a sinusoid whose phase
is $\omega\log(1/b)$ and whose amplitude is nonzero because $Q\ne0$.

\begin{lemma}[Uniform inner--outer matching]\label{lem:matching}
There is a function $\varepsilon(b)\to0$ such that, uniformly for
$|\mu|\le\mu_0$,
\begin{align}
 F_2(b,\mu)&=\mathcal E(\mu)+O(\varepsilon(b)),\label{eq:F2}\\
 b^{-5/2}F_1(b,\mu)
 &=\Re\!\left(C_o(\mu)Qe^{i\omega\log(1/b)}\right)
   +O(\varepsilon(b)).\label{eq:F1}
\end{align}
With $\delta_b=b^{1/2}$ one may take
\begin{equation}\label{eq:rho-rate}
 \varepsilon(b)
 =C\left[
   \kappa_b+\delta_b^\sigma+\delta_b^2
   +\left(\frac b{\delta_b}\right)^5
 \right](1+|\log\delta_b|).
\end{equation}
In particular \(\varepsilon(b)\to0\) as \(b\downarrow0\).  Indeed,
\(\beta_+=(-5+3\sqrt5)/2\in(0,2)\),
\(\rho(b)=O(b^{\beta_+/m})\), and with
\(\delta_b=b^{1/2}\), \(\tau_b=b^{-1/2}\), every power occurring in
\(\kappa_b\) and \((b/\delta_b)^5\) is strictly positive; multiplication by
\(1+|\log\delta_b|\) therefore still tends to zero.
\end{lemma}

\begin{proof}
\emph{Step 1: match the odd coefficient at the overlap.}
At $t=\delta_b$, use the convention that a real odd field with complex
coefficient $A_b$ is written as $\Re(A_b\Psi_{o,\widehat\mu})$, and match
the inner odd field in this convention.  The weighted inner estimate and
\eqref{eq:outer-Frobenius} give
\[
 A_b
 =\delta_b^{5/2-i\omega}Q\tau_b^{-5/2+i\omega}
  \left(1+O(\tau_b^{-\eta_*})+O(\kappa_b)+O(\delta_b^\sigma)\right).
\]
The scale and phase cancellation is exact:
\begin{align*}
 \delta_b^{5/2-i\omega}\tau_b^{-5/2+i\omega}
 &=\delta_b^{5/2-i\omega}
   \left(\frac{\delta_b}{b}\right)^{-5/2+i\omega}\\
 &=b^{5/2-i\omega}
  =b^{5/2}e^{i\omega\log(1/b)}.
\end{align*}
Thus the arbitrary overlap scale disappears from the leading term.  This is
the source of the full $\omega\log(1/b)$ phase.

\emph{Step 2: propagate from the overlap to a fixed outer section.}
Put $A=b^{5/2}$.  On the outer matching corridor,
\[
 |z(t)|\le CA t^{-5/2},\qquad |z'(t)|\le CA t^{-7/2}.
\]
Lemma~\ref{lem:outer-parity} gives
\[
 L_{\widehat\mu} z
 =
 -\bigl(H(E,z)-H(E_{\widehat\mu},0)\bigr)z'
 +\bigl(a(E,z)-a(E_{\widehat\mu},0)\bigr)z,
\]
and therefore
\begin{equation}\label{eq:outer-odd-remainder}
 |L_{\widehat\mu} z(t)|
 \le C\vartheta_b A t^{-9/2}+CA^3t^{-19/2}.
\end{equation}
The first term is the weighted even-background mismatch acting linearly on
the odd field.  The second is the genuine odd nonlinearity and is cubic by
Weyl parity.

The operator \(L_{\widehat\mu}\) is the variable-coefficient odd
linearization about the nearby background \(G_{\widehat\mu}\), not the
sine-cone operator \eqref{eq:odd-sine}.  Regular-singular perturbation theory
gives for its Frobenius pair a Wronskian \(O(t^{-6})\), uniformly for bounded
\(\mu\), and the Green kernel from the singular end to a fixed outer point has
weight \(O(t^{7/2})\).  Thus
\begin{align*}
 A^{-1}\int_{\delta_b}^{t_*}
 O(t^{7/2})O(\vartheta_b A t^{-9/2})\,dt
 &\le C\vartheta_b(1+|\log\delta_b|),\\
 A^{-1}\int_{\delta_b}^{t_*}
 O(t^{7/2})O(A^3t^{-19/2})\,dt
 &\le CA^2\delta_b^{-5}
 =C\left(\frac b{\delta_b}\right)^5.
\end{align*}
Together with Lemma~\ref{lem:weighted-inner} and the overlap Frobenius
error, this proves the normalized odd estimate at the fixed regular section
$t=t_*$.  From $t_*$ to the compact corridor containing the $H=0$ events the
Einstein equation is a regular ODE.  Its flow map is uniformly $C^2$ on the
fixed positive neighborhood containing all $G_\nu$, $|\nu|\le2\mu_0$.
The regular flow and the fixed-section odd detector are Weyl equivariant.
Write the corridor state as $(E,V)$, with $E$ even and $V$ odd, and let
$\mathscr D(E,V)$ denote the value of $F_1$ at the fixed regular section.
Then
\[
 \mathscr D(E,-V)=-\mathscr D(E,V),\qquad \mathscr D(E,0)=0.
\]
On the fixed positive neighborhood the map $\mathscr D$ is uniformly $C^2$.
Hadamard's formula therefore factors it as
\[
 \mathscr D(E,V)=\mathscr M(E,V)V,
\]
with \(\mathscr M\) uniformly \(C^1\) and even in \(V\).  Since
\(\|E-E_{\widehat\mu}\|=O(\vartheta_b)\) and \(|V|=O(A)\) at the fixed
section, the change of the multiplier caused by the even background is
\(O(\vartheta_b)\), hence contributes \(O(A\vartheta_b)\).  Oddness gives no
constant term, and the non-linear odd correction is at least quadratic in the
small corridor state; it is therefore \(O(A^2)\) (in fact the pure odd
feedback is cubic).  Thus
\[
 \mathscr D(E,V)
 =\mathscr D(E_{\widehat\mu},V_{\rm lin})
  +O(A\vartheta_b)+O(A^2).
\]
Thus a pure even/background error never produces a naked
$O(\vartheta_b)$ contribution to $F_1$; Weyl parity supplies the essential
factor $A=b^{5/2}$.  After division by $A$ the corridor error is
$O(\vartheta_b)+O(A)$, and is absorbed by \eqref{eq:rho-rate}.  The same
regular propagation extends \eqref{eq:weighted-even-state} to the whole
corridor.  Consequently the fixed-corridor form of \eqref{eq:F1} holds with
the same $\varepsilon(b)$.

\emph{Step 3: move the detector to the first $H=0$ section.}
Let \(T_{\widehat\mu}\) be the first \(H=0\) event of
\(G_{\widehat\mu}\), and let \(T=T(b,\mu)\) be the corresponding event of the
exact solution.  These are moving event times, not fixed copies of
\(T_0=\pi/2\).  The propagated even estimate gives
\[
 |H_{b,\mu}-H_{\widehat\mu}|\le C\varepsilon(b)
\]
on a fixed neighborhood of $T_{\widehat\mu}$.  Since
\[
 H'=-2(L_1^2+L_2^2+L_3^2)-6\le-6,
\]
the event is uniformly transverse and
\begin{equation}\label{eq:event-shift}
 |T-T_{\widehat\mu}|\le C\varepsilon(b).
\end{equation}
On the same neighborhood, with \(F_1(t):=f_1(t)-f_2(t)\) evaluated along
the exact trajectory, outer propagation gives
\[
 |F_1(t)|+|\partial_tF_1(t)|\le CA.
\]
Hence
\[
 b^{-5/2}|F_1(T)-F_1(T_{\widehat\mu})|
 \le C\varepsilon(b).
\]
Thus the moving first-$H=0$ section preserves the normalized odd asymptotic.
For the even detector, $L_3$ and $L_3'$ are uniformly bounded there, so
\eqref{eq:event-shift} changes $F_2$ by $O(\varepsilon(b))$.  The symmetric
outer solution contributes $\mathcal E(\widehat\mu)$ to $F_2$, while the odd
contribution to the even variables is quadratic.  Equations
\eqref{eq:E-muhat} and \eqref{eq:Co-muhat} convert the two leading responses
back from $\widehat\mu$ to $\mu$, with an $O(\kappa_b)$ error already
contained in \eqref{eq:rho-rate}.  This proves \eqref{eq:F2} as well as
\eqref{eq:F1}.
\end{proof}

In \eqref{eq:rho-rate}, the displayed error groups collect seven decay
mechanisms: the threshold inverse modulus, the inner stable tail, unstable
detuning transport, positive-Einstein even normalization, the outer
Frobenius remainder, relative positive-Einstein odd shadowing, and
parity-controlled odd nonlinear feedback.  The logarithm is the borderline
outer Green loss.  No constant in these estimates depends on
\(\mu\in[-\mu_0,\mu_0]\): the inner constants come from one fixed cone chart,
the regular-singular outer constants are uniform on the compact
\(\nu\)-interval, and the final corridor is fixed and uniformly transverse.
Thus the error is genuinely uniform on the whole Poincar\'e--Miranda
rectangle, not merely along a sequence of detunings.  Since
\(\rho(b)=O(b^{\beta_+/m})\), all terms tend to zero for
\(\delta_b=b^{1/2}\).

\begin{lemma}[Positivity up to the reflection section]\label{lem:positivity}
After decreasing $b_0$ and $\mu_0$ if necessary, every solution used in
Lemma~\ref{lem:matching} is regular on the interval from the left focal orbit
to its first $H=0$ event and satisfies
\[
 f_i(t)>0\qquad (0<t\le T(b,\mu),\ i=1,2,3).
\]
\end{lemma}

\begin{proof}
\emph{Inner corridor.}
We make the uniform positivity margin explicit.  Work first in the inner
variables $t=b\tau$, $f_i=bF_i$.  Choose a fixed $\tau_0>1$.  Smooth focal
expansions and continuous dependence on the focal parameter give, uniformly
for $|\mu|\le\mu_0$ and small $b$,
\[
 F_i(\tau)>0\qquad(0<\tau\le\tau_0,\ i=1,2,3).
\]
For $\tau\ge\tau_0$, the threshold trajectory lies in the positive
projective chart and converges to the positive cone $p_2$.  Increasing
$\tau_0$ if necessary, there is therefore $c_{\rm pos}>0$ such that along the
threshold
\[
 \frac{F_{i,*}(\tau)}{\tau}\ge 4c_{\rm pos}
 \qquad(\tau\ge\tau_0,\ i=1,2,3).
\]
The detuned Ricci-flat orbit converges to the threshold uniformly on the
inner corridor.  Its stable-entry error is $O(\rho(b))$ and its unique
unstable even component is at most $O(\delta_b^{\beta_+})$ there; the
remaining stable tail is $o(1)$.  Lemma~\ref{lem:absolute-inner} gives the
positive-Einstein perturbation $O(b^2e^{2s})$, which is
$O(\delta_b^2)$ at $s=s_b$.
After decreasing $b_0$ these errors are smaller than $c_{\rm pos}$ in the
normalized projective variables.  Consequently
\[
 \frac{f_i(t)}{t}\ge c_{\rm pos}
 \qquad(b\tau_0\le t\le\delta_b,\ i=1,2,3).
\]
Together with the focal interval this proves positivity throughout the inner
region and, in particular, gives a constant $c_{\rm in}>0$, independent of
$b$ and $\mu$, such that
\begin{equation}\label{eq:overlap-positive-margin}
 f_i(\delta_b)\ge c_{\rm in}\delta_b,\qquad i=1,2,3.
\end{equation}

\emph{Outer corridor.}
For the outer interval use the weighted variables from
Lemma~\ref{lem:outer-parity},
\[
 x_e=\frac{\sqrt{f_1f_2}}{t},\qquad y_e=\frac{f_3}{t},
 \qquad z=\log\frac{f_1}{f_2}.
\]
For $|\nu|\le2\mu_0$, the symmetric family $G_\nu$ depends continuously on
$\nu$ and is positive on a fixed corridor containing all its first $H=0$
events.  Hence, after shrinking $\mu_0$, there is $c_{\rm out}>0$ such that
\begin{equation}\label{eq:outer-reference-margin}
 x_{e,\nu}(t)\ge2c_{\rm out},\qquad y_{e,\nu}(t)\ge2c_{\rm out}
\end{equation}
throughout that corridor.

The actual matched solution satisfies, by
\eqref{eq:weighted-even-state},
\[
 |x_e-x_{e,\widehat\mu}|+|y_e-y_{e,\widehat\mu}|
 \le C\vartheta_b,
\]
while the odd Frobenius estimate gives for $t\ge\delta_b$
\[
 |z(t)|\le Cb^{5/2}t^{-5/2}
 \le Cb^{5/4}.
\]
Since $\vartheta_b\to0$ (and hence also $\varepsilon(b)\to0$), choose $b_0$
so that
\[
 C\vartheta_b<c_{\rm out},\qquad Cb^{5/4}<1
 \qquad(0<b<b_0).
\]
Then \eqref{eq:outer-reference-margin} yields
$x_e,y_e\ge c_{\rm out}>0$.  Because
\[
 f_1=tx_e e^{z/2},\qquad
 f_2=tx_e e^{-z/2},\qquad
 f_3=ty_e,
\]
all three scale factors remain strictly positive on the entire outer
corridor.

\emph{Moving reflection section.}
Finally \eqref{eq:event-shift} gives
$|T(b,\mu)-T_{\widehat\mu}|\le C\varepsilon(b)$, so for smaller $b_0$ the
moving first-$H=0$ event lies inside this same fixed positive corridor.
Thus no scale factor can vanish before the moving reflection section.  Since the
Einstein system is a regular ODE whenever $f_1f_2f_3>0$, the solution is
regular on the whole half-interval.
\end{proof}

\begin{remark}[Continuation to the moving reflection section]\label{rem:no-circular-positivity}
The inner and outer matching estimates are first obtained on the fixed
positive coordinate corridors on their maximal interval of validity.  Their
constants depend only on those fixed corridors.  The quantitative lower
bounds in Lemma~\ref{lem:positivity} show that, for sufficiently small $b$,
the solution stays a uniform distance from every face $f_i=0$ up to the
moving $H=0$ event.  Hence the maximal interval of validity reaches that
event.  Thus the local matching estimates extend throughout the interval on
which the reflection detector is evaluated.
\end{remark}

\section{Closing the metric: repeated phase rectangles}

The matching asymptotics reduce closure to a two-dimensional sign argument.
Choose \(\mu_0>0\) small enough that, for some \(m_e>0\),
\begin{equation}\label{eq:even-face-margin}
 \mathcal E(-\mu_0)\le-m_e<0<m_e\le\mathcal E(\mu_0).
\end{equation}
Writing
\[
 C_o(0)Q=Re^{i\phi},\qquad R>0,
\]
decrease $\mu_0$ further so that
\begin{equation}\label{eq:odd-coefficient-margin}
 |C_o(\mu)Q-C_o(0)Q|\le R/4
 \qquad(|\mu|\le\mu_0).
\end{equation}
For every integer \(k\), define the consecutive half-period endpoints
\[
 b_k^{\mathrm{R}}
 =\exp\!\left(-\frac{k\pi-\phi}{\omega}\right),
 \qquad
 b_k^{\mathrm{L}}
 =\exp\!\left(-\frac{(k+1)\pi-\phi}{\omega}\right).
\]
Thus
\[
 b_k^{\mathrm{L}}<b_k^{\mathrm{R}},\qquad
 b_k^{\mathrm{L}}=b_{k+1}^{\mathrm{R}},\qquad
 b_k^{\mathrm{R}}\longrightarrow0,
\]
and the phase runs through one full half-period on each interval
\([b_k^{\mathrm{L}},b_k^{\mathrm{R}}]\).  Set
\[
 \mathcal R_k=[b_k^{\mathrm{L}},b_k^{\mathrm{R}}]
 \times[-\mu_0,\mu_0].
\]
The focal branch, the regular ODE flow, and the transverse first-\(H=0\)
hitting time depend continuously on \((b,\mu)\); hence the exact detector
\(\mathcal F\) in \eqref{eq:detector-map} is continuous on every sufficiently
late closed rectangle \(\mathcal R_k\).

Fix \(C_{\rm det}\) so that both remainders in
\eqref{eq:F2}--\eqref{eq:F1} are bounded by
\(C_{\rm det}\varepsilon(b)\), uniformly for \(|\mu|\le\mu_0\).
By construction,
\[
 \phi+\omega\log(1/b_k^{\mathrm{R}})=k\pi,
 \qquad
 \phi+\omega\log(1/b_k^{\mathrm{L}})=(k+1)\pi.
\]
Hence the leading normalized odd responses on the two vertical faces are
\((-1)^kR\) and \((-1)^{k+1}R\).  By
\eqref{eq:odd-coefficient-margin} their absolute values remain at least
\(3R/4\), with opposite signs, uniformly in \(|\mu|\le\mu_0\).

Because \(\varepsilon(b)\to0\), there exists \(\bar b>0\) such that
\[
 C_{\rm det}\varepsilon(b)<\min\{R/4,m_e/2\}
 \qquad(0<b<\bar b);
\]
no monotonicity of \(\varepsilon\) is used.  Choose \(K\) so large that, for
every \(k\ge K\), \(\mathcal R_k\subset(0,\bar b)\times[-\mu_0,\mu_0]\) and
the whole rectangle lies in the inverse-coordinate, positivity and uniform
matching domains.  Then \(F_1\) has opposite signs on the two \(b\)-faces of
\(\mathcal R_k\).  Likewise, \eqref{eq:even-face-margin} and \eqref{eq:F2}
give
\[
 F_2(b,-\mu_0)<0<F_2(b,\mu_0)
 \qquad(b_k^{\mathrm{L}}\le b\le b_k^{\mathrm{R}}).
\]
Thus on the two \(b\)-faces of \(\mathcal R_k\), \(F_1\) has opposite
signs, while on the two \(\mu\)-faces \(F_2\) has opposite signs.  By the
two-dimensional intermediate value theorem just described
(Poincar\'e--Miranda), these opposite-face sign conditions give
\[
 (b_k,\mu_k)\in\mathcal R_k,\qquad
 \mathcal F(b_k,\mu_k)=0
\]
for every \(k\ge K\).  Since \(F_1\) is nonzero on every shared vertical
boundary, these zeros lie in the interiors of their half-period cells; in
particular they are distinct and \(b_k\to0\).

Each detector zero yields the conditions for a smooth Weyl
reflection across the corresponding principal orbit.  The resulting double has the standard adjoint \(SU(3)\) group
diagram on \(S^7\).

\medskip
\noindent\textbf{What has been proved at this point.}
The Ricci-flat analysis supplies a genuine threshold trajectory and a
nonzero oscillatory coefficient \(Q\); the matching theorem transfers its
phase to the \emph{exact} positive-Einstein detector with an error tending to
zero uniformly in the detuning parameter.  The fixed sign margins on the four
faces of each late rectangle are therefore sign statements for the exact
closing equations themselves.  In particular, the existence theorem is
entirely analytic.

\begin{proposition}[Logarithmic quantization]\label{prop:quantization}
Let \(\omega=\sqrt{15}/2\).  The detector zeros above may be chosen, one in
each sufficiently late logarithmic half-period, so that for some
\(\theta_*\in\mathbb R/\pi\mathbb Z\),
\[
 \omega\log\frac1{b_k}=k\pi+\theta_*+o(1).
\]
Moreover \(\mu_k\to0\), and
\[
 \frac{b_{k+1}}{b_k}\longrightarrow
 e^{-\pi/\omega}=e^{-2\pi/\sqrt{15}}.
\]
\end{proposition}

\begin{proof}
Shrink \(\mu_0\), if necessary, so that the expansion
\(\mathcal E(\mu)=C_e\mu+O(\mu^2)\), with \(C_e>0\), makes
\(\mathcal E\) one-to-one on \([-\mu_0,\mu_0]\) and gives
\(|\mathcal E(\mu)|\ge c_e|\mu|\) there for some \(c_e>0\).
At a detector zero, \eqref{eq:F2} therefore implies
\[
 |\mu_k|\le C\varepsilon(b_k)=o(1).
\]
Consequently the Lipschitz dependence of \(C_o\) gives
\[
 C_o(\mu_k)Q=Re^{i\phi}+o(1).
\]
Using \eqref{eq:F1} at the same zero yields
\[
 \cos\!\left(\omega\log\frac1{b_k}+\phi\right)=o(1).
\]
The phase of \(b_k\) lies in the \(k\)-th half-period
\([k\pi,(k+1)\pi]\).  The cosine has exactly one zero in that interval,
namely \((k+\tfrac12)\pi\); hence
\[
 \omega\log\frac1{b_k}+\phi
 =k\pi+\frac\pi2+o(1).
\]
Thus the asserted phase law holds with
\(\theta_*=\pi/2-\phi\pmod\pi\).  Subtracting the formulas for consecutive
indices gives
\[
 \omega\log\frac{b_k}{b_{k+1}}=\pi+o(1),
\]
which proves the ratio limit.
\end{proof}

\begin{remark}[Normalization ledger]\label{rem:normalization-ledger}
There are three distinct operations in the proof.  The choice
\(g=dt^2+g_t\) fixes the unit-speed normal coordinate; smoothness then forces
\(f_1'(0)=1\).  The equation \(\Ric(g)=6g\) fixes the overall homothety scale
of the compact Einstein metric.  Finally, \(t=b\tau,\ f_i=bF_i\) replaces
\(g\) by the blown-up metric \(b^{-2}g\) in the inner analysis and changes the
Einstein constant from \(6\) to \(6b^2\).  None of these operations prescribes
the first \(H=0\) time \(T(b,\mu)\) or the orbit-space length \(2T(b,\mu)\).
\end{remark}

\begin{proof}
Lemma~\ref{lem:positivity} shows that every detector zero determines a
regular positive Einstein half-metric.  At such a
zero, the first $H=0$ orbit satisfies
\[
 f_1=f_2,\qquad L_3=0.
\]
Since \(H=0\), also \(L_1+L_2=0\).  Because \(f_1=f_2>0\) at the
reflection orbit,
\[
 f_1'=-f_2',\qquad f_3'=0.
\]
These are exactly the \(C^1\) fixed-point conditions for the Weyl reflection.
Lemma~\ref{lem:positivity} gives \(f_i(T)>0\), so this orbit is principal and
the Einstein ODE is regular there.  Writing \(s=t-T(b,\mu)\) and abbreviating
\(T=T(b,\mu)\), define for
\(0\le s\le T\)
\[
 f_1(T+s)=f_2(T-s),\qquad
 f_2(T+s)=f_1(T-s),\qquad
 f_3(T+s)=f_3(T-s).
\]
At \(s=0\) both the metric coefficients and their first derivatives match
after the Weyl identification.  Since \(f_i(T)>0\), the Einstein equations are
a regular analytic ODE system in a neighborhood of \(t=T\).  The reflected
triple satisfies the same system and has the same Cauchy data at \(t=T\).
Ordinary ODE uniqueness therefore identifies it with the original local
continuation; in particular the reflected metric is not merely \(C^1\), but
smooth (indeed real analytic in the regular variables) across the principal
orbit \(t=T\).  At the right endpoint \(t=2T\),
\[
 f_2(2T)=f_1(0)=0,\qquad
 f_1(2T)=f_3(2T)=b.
\]
If \(\rho=2T-t\) is inward arclength from the right dsingular orbit, then the
reflection formulas give
\[
 f_2(2T-\rho)=f_1(\rho),\qquad
 f_1(2T-\rho)=f_2(\rho),\qquad
 f_3(2T-\rho)=f_3(\rho).
\]
Thus, as functions of the local signed radial variable \(\rho\), the right
endpoint satisfies the same smooth singular-orbit parity as the left:
the collapsing factor is odd with unit first derivative and the two
noncollapsing factors are exchanged.  No additional endpoint condition is
required.

For completeness, the topology can be read directly from the action.  Regard
$S^7$ as the unit sphere in the eight-dimensional adjoint representation
$\mathfrak{su}(3)$.  A regular element has stabilizer $T^2$, while an endpoint
of a Weyl chamber has stabilizer $U(2)$ and orbit $\mathbb{CP}^2$.  The two singular isotropy groups correspond to two distinct roots and satisfy
\(U(2)_\pm/T^2\cong S^2\), so both caps are honest rank-three disk bundles
over \(\mathbb{CP}^2\).  The embedded group diagram
\[
 T^2\subset U(2)_-,U(2)_+\subset SU(3)
\]
has \(U(2)_-\) and \(U(2)_+\) corresponding to the two endpoint walls of an
\(A_2\) Weyl chamber.  The Weyl element used in the reflection exchanges
these two walls and the associated first two isotropy summands.  Consequently
the double
\[
 SU(3)\times_{U(2)_-}D^3
 \ \cup_{\,SU(3)/T^2}\
 SU(3)\times_{U(2)_+}D^3
\]
with this gluing is precisely the standard adjoint cohomogeneity-one manifold
\(S(\mathfrak{su}(3))\cong S^7\).  Thus no alternative double-disk-bundle or
orbifold topology is introduced by the reflection.  Finally the phase rectangles tend to \(b=0\),
hence \(b_n\to0\).  The focal anisotropy of the corresponding solution is
\[
 \beta_n=c(b_n,\mu_n).
\]
By the uniform detuning estimate \eqref{eq:rho-detune},
\[
 |\beta_n-c_*|\le \rho(b_n)\longrightarrow0.
\]
The threshold is the far endpoint of the entrance interval on the chosen
side \(c<0\), strictly away from the Bryant--Salamon value \(c=0\); hence
\(c_*<0\).  Therefore \(\beta_n\to c_*<0\), and in particular
\(\beta_n\ne0\) for every sufficiently large \(n\).  The final assertion,
pairwise nonisometry after passage to a subsequence, follows from the
curvature estimate in the next section.
\end{proof}

Proposition~\ref{prop:quantization} shows that the half-turn scale is not
merely a feature of the comparison rectangles: one may choose an exact
Einstein closing in every sufficiently late half-period, and these exact
closing scales satisfy
\[
 \frac{b_{k+1}}{b_k}\longrightarrow
 e^{-\pi/\omega}=e^{-2\pi/\sqrt{15}}.
\]
No uniqueness of the nonlinear detector zero inside a half-period is asserted
or needed for this conclusion.

\section{Geometry of the quantized sequence}\label{sec:geometry-sequence}

We finish by identifying the two geometric scales of the sequence constructed
above.  The Ricci-flat threshold is not merely an auxiliary limiting object
used to create detector sign changes: after blow-up it is exactly the metric
that bubbles off at each singular orbit.  Concretely, the curvature
concentrates in neighborhoods of radius comparable to $b_n$, and rescaling
those neighborhoods by $b_n^{-2}$ produces a complete noncompact Ricci-flat
limit.  On the unrescaled scale the same sequence converges away from the
endpoints to the singular sine cone over the nonnormal Wallach metric.

We first record a curvature estimate that will also sharpen the
nonisometry argument.

\begin{lemma}[Curvature in a controlled cone corridor]
\label{lem:cone-curvature-bound}
Let
\[
 g=dt^2+\sum_{i=1}^3 f_i(t)^2Q|_{\mathfrak m_i}
\]
be an $SU(3)$-invariant diagonal metric satisfying $\Ric(g)=6g$ on an
interval $I\subset(0,T]$.  Suppose that, for constants $c_0,C_0>0$,
\begin{equation}\label{eq:cone-corridor-bounds}
 c_0t\le f_i(t)\le C_0t,
 \qquad |tL_i(t)|\le C_0,
 \qquad i=1,2,3.
\end{equation}
Then
\[
 |\Rm(g)|(t)\le C t^{-2}\qquad(t\in I),
\]
where $C$ depends only on $c_0,C_0,T$ and the fixed homogeneous
normalization $Q$.
\end{lemma}

\begin{proof}
The radial sectional curvatures are
\[
 K(\partial_t,\mathfrak m_i)=-\frac{f_i''}{f_i}=-L_i'-L_i^2.
\]
The Einstein equations give
\[
 L_i'=r_i-6-HL_i,
 \qquad H=2(L_1+L_2+L_3).
\]
By \eqref{eq:cone-corridor-bounds}, $|H|\le C/t$ and hence
$|HL_i|+L_i^2\le C/t^2$.  Writing $f_i=tu_i$, the Wallach orbit-Ricci
formula \eqref{eq:intro-orbit-ricci} becomes
\[
 r_i=t^{-2}\mathcal R_i(u_1,u_2,u_3),
\]
where $\mathcal R_i$ is smooth on the compact set
$[c_0,C_0]^3$.  Thus $|r_i|\le Ct^{-2}$ and
\[
 |L_i'|+L_i^2\le Ct^{-2}
\]
after increasing $C$ using $t\le T$.

For a two-plane tangent to a principal orbit, the Gauss equation expresses
its ambient curvature as its curvature in the homogeneous orbit metric
$g_t$ plus a quadratic expression in the principal curvatures $L_i$.
The latter is $O(t^{-2})$.  Since
\[
 g_t=t^2\sum_i u_i^2Q|_{\mathfrak m_i},
\]
every sectional curvature of $g_t$ is $t^{-2}$ times a smooth function of
$(u_1,u_2,u_3)$ on $[c_0,C_0]^3$, and is therefore $O(t^{-2})$ as well.
The radial and tangential estimates give the assertion.
\end{proof}

We next record the focal coefficient.  The smooth focal expansions are
\[
 f_1(t)=t+O(t^3),\qquad
 f_2(t)=b+\beta t+b_2t^2+O(t^3),\qquad
 f_3(t)=b-\beta t+b_2t^2+O(t^3).
\]
Using $L_2'=r_2-6-HL_2$ and the orbit-Ricci formula gives
\[
 r_2=\frac{2\beta}{bt}+\frac{3}{2b^2}+O(t),
 \qquad
 HL_2=\frac{2\beta}{bt}
      +\frac{4b_2}{b}-\frac{2\beta^2}{b^2}+O(t).
\]
The singular terms cancel, and comparison with
$L_2'=2b_2/b-\beta^2/b^2$ yields
\[
 b_2=-b+\frac{\beta^2}{2b}+\frac1{4b}.
\]
Consequently, for the radial plane through $\mathfrak m_2$,
\begin{equation}\label{eq:focal-curvature-exact}
 K_{02}(0)=-\frac{f_2''(0)}{f_2(0)}
 =2-\frac{1+2\beta^2}{2b^2}.
\end{equation}

\begin{proposition}[Sharp curvature scale]\label{prop:sharp-curvature-scale}
For the sequence $g_n$ of Theorem~\ref{thm:main},
\[
 \sup_{S^7}|\Rm(g_n)|\asymp b_n^{-2}.
\]
More precisely, at either singular orbit,
\begin{equation}\label{eq:focal-curvature-limit}
 b_n^2K_{02}^{(n)}\longrightarrow-\frac{1+2c_*^2}{2}.
\end{equation}
Hence
\begin{equation}\label{eq:curvature-quantization}
 \frac{|K_{02}^{(n+1)}|}{|K_{02}^{(n)}|}
 \longrightarrow e^{4\pi/\sqrt{15}}.
\end{equation}
\end{proposition}

\begin{proof}
Equation \eqref{eq:focal-curvature-exact} and $\beta_n\to c_*$ give
\eqref{eq:focal-curvature-limit}, and hence the lower bound
$\sup|\Rm(g_n)|\ge c b_n^{-2}$.

For the upper bound consider first one half of the metric.  On a fixed
rescaled focal interval $0\le t\le \tau_0b_n$, put $t=b_n\tau$ and
$f_i=b_nF_i$.  The smooth focal construction, the convergence of the focal
parameter to $c_*$, and ordinary ODE dependence on every fixed $\tau$-interval
give a uniform curvature bound for the rescaled metrics, hence
$|\Rm(g_n)|\le Cb_n^{-2}$ there.

On the inner cone corridor
\[
 \tau_0b_n\le t\le\delta_{b_n},\qquad \delta_b=b^{1/2},
\]
the proof of Lemma~\ref{lem:positivity}, together with the cone-chart
shadowing estimates, gives uniformly
\[
 c_0\le \frac{f_i(t)}t\le C_0,
 \qquad |tL_i(t)|\le C_0.
\]
The upper bounds follow from the same compact cone chart that gives the lower
positivity margin: in the notation of Section~6, the weighted orbit factors
$f_i/t$ and the weighted logarithmic derivatives $tL_i$ remain in a fixed
compact neighborhood of the threshold cone.  Lemma~\ref{lem:cone-curvature-bound}
therefore gives
\[
 |\Rm(g_n)|(t)\le Ct^{-2}\le Cb_n^{-2}.
\]

On the outer corridor $\delta_{b_n}\le t\le T_n$, use the weighted
variables of Lemma~\ref{lem:outer-parity}.  The state $E$ there contains
$x_e,y_e$ and the corresponding weighted logarithmic derivatives.  Hence
\eqref{eq:weighted-even-state}, regular dependence of the symmetric family
$G_\nu$, and the odd first-order estimate imply uniformly
\[
 c_1\le x_e,y_e\le C_1,
 \qquad |tL_i|\le C_1,
 \qquad |z|\le C_1b_n^{5/2}t^{-5/2}.
\]
Since
\[
 f_1=tx_e e^{z/2},\qquad
 f_2=tx_e e^{-z/2},\qquad
 f_3=ty_e,
\]
we obtain, after increasing $C_1$,
\[
 c_2t\le f_i(t)\le C_2t,
 \qquad |tL_i(t)|\le C_2
\]
throughout the outer corridor.  Lemma~\ref{lem:cone-curvature-bound} therefore
applies there as well and gives
\[
 |\Rm(g_n)|(t)\le C\bigl(1+t^{-2}\bigr)
 \le C\delta_{b_n}^{-2}=Cb_n^{-1}\le Cb_n^{-2}.
\]
The midpoint Weyl reflection gives the same estimate on the second half.
This proves the upper bound.  Finally
\eqref{eq:curvature-quantization} follows from
\eqref{eq:focal-curvature-limit} and
$b_{n+1}/b_n\to e^{-2\pi/\sqrt{15}}$.
\end{proof}

\begin{proposition}[Two-ended Ricci-flat bubbling]\label{prop:two-ended-bubbling}
Let $p_n^-$ lie on the initial $\mathbb{CP}^2$ singular orbit and let $p_n^+$
lie on the opposite singular orbit.  Then, after the natural equivariant
identifications,
\[
 (S^7,b_n^{-2}g_n,p_n^\pm)
 \longrightarrow(M_*,g_{\RF,*},p_*)
\]
smoothly on compact subsets, where $g_{\RF,*}$ is the complete Ricci-flat
threshold metric selected in Proposition~\ref{prop:threshold}.  The two
pointed limits are isometric, up to the Weyl interchange of the first two
isotropy summands.
\end{proposition}

\begin{proof}
Write the exact closing parameters as $(b_n,\mu_n)$.  The proof of
Proposition~\ref{prop:quantization} gives $\mu_n\to0$, while
\eqref{eq:rho-detune} gives $\beta_n\to c_*$.  In the inner variables
\[
 t=b_n\tau,\qquad f_i(t)=b_nF_{i,n}(\tau),
\]
the Einstein equation is the Ricci-flat equation plus a perturbation of size
$b_n^2$.  On every fixed $\tau$-interval the smooth focal germs therefore
converge to the Ricci-flat germ with parameter $c_*$; equivalently, this is
the fixed-interval version of Lemma~\ref{lem:absolute-inner}.  On compact
sets separated from the singular orbit, standard ODE bootstrapping upgrades
this to $C^\infty$ convergence.  At the singular orbit itself, the smooth
focal-coordinate construction and its smooth dependence on the focal
parameters give the same $C^\infty$ convergence.  Thus
\[
 (S^7,b_n^{-2}g_n,p_n^-)
 \longrightarrow(M_*,g_{\RF,*},p_*)
\]
pointedly and smoothly.  Proposition~\ref{prop:threshold} supplies
completeness of the limit.

At a detector zero the conditions $H=0$, $f_1=f_2$ and $L_3=0$ imply
$L_1=-L_2$.  These are exactly the fixed-point conditions for the Weyl
reflection used in the closing construction.  The resulting reflection is
an isometry of $g_n$ interchanging the two singular orbits and the first two
isotropy summands.  Taking $p_n^+$ to be the reflected point proves the same
pointed limit at the opposite end.
\end{proof}

\begin{proposition}[Singular sine-cone limit]\label{prop:sine-cone-limit}
Let $h_2$ be the nonnormal Einstein metric on $SU(3)/T^2$ determined by
\[
 f_1=f_2=\frac1{\sqrt5},\qquad f_3=\sqrt{\frac25}.
\]
Then, on the regular part and after the natural equivariant identification,
\begin{equation}\label{eq:sine-cone-limit}
 g_n\longrightarrow g_{\mathrm{sc}}:=dt^2+\sin^2t\,h_2
\end{equation}
in $C^\infty_{\mathrm{loc}}((0,\pi)\times SU(3)/T^2)$.  The metric completion
of $g_{\mathrm{sc}}$ has two conical singularities, and its tangent cone at
either end is the asymptotic cone of $g_{\RF,*}$.
\end{proposition}

\begin{proof}
Fix first a compact interval $J\Subset(0,\pi/2)$.  For the exact closing sequence,
$\mu_n\to0$, and \eqref{eq:muhat} gives $\widehat\mu_n\to0$.  Hence the
symmetric outer backgrounds $G_{\widehat\mu_n}$ converge smoothly on every
fixed regular corridor to $G_0$, which is the $p_2$ sine cone.  The weighted
even estimate \eqref{eq:weighted-even-state} tends to zero, while the odd
outer estimate has size $O(b_n^{5/2}t^{-5/2})$ and therefore tends to zero
with all derivatives on $J$.  Regular ODE dependence then gives smooth
convergence of the complete metric coefficients to those of $G_0$ on $J$.
The event estimate \eqref{eq:event-shift} and
$T_{\widehat\mu_n}\to T_0=\pi/2$ give $T_n\to\pi/2$.

After reflection, identify the regular part of $(S^7,g_n)$ with
\[
 (0,2T_n)\times SU(3)/T^2,
\]
using the original unit-speed radial coordinate on the left half and its
reflected continuation on the right.  Since $T_n\to\pi/2$, every compact
interval $J\Subset(0,\pi)$ is contained in $(0,2T_n)$ for all sufficiently
large $n$.  If $J$ stays a positive distance from $\pi/2$, the preceding
left-hand convergence and its reflected version give smooth convergence on
$J$.  If $J$ meets $\pi/2$, choose a fixed small regular neighborhood of the
sine-cone state at $T_0=\pi/2$.  The matched states at $T_n$ converge to that
state, and the Einstein vector field is smooth on this neighborhood; uniform
ODE dependence, followed on the right by the exact Weyl reflection, therefore
gives convergence with all derivatives on a fixed neighborhood of $\pi/2$.
A finite covering of $J$ proves \eqref{eq:sine-cone-limit} in
$C^\infty_{\mathrm{loc}}((0,\pi)\times SU(3)/T^2)$.  Since $\Ric(h_2)=5h_2$, $g_{\mathrm{sc}}$ has Einstein constant
$6$.  Near either endpoint,
\[
 g_{\mathrm{sc}}
 =dt^2+t^2h_2+O(t^4),
\]
so its completion is conical there.  Proposition~\ref{prop:threshold}
identifies the asymptotic cone of $g_{\RF,*}$ as
$dr^2+r^2h_2$, proving the final assertion.
\end{proof}

The preceding propositions give a two-scale description of the whole
sequence:
\[
 \begin{gathered}
 \text{Ricci-flat threshold bubble}
 \quad\longleftrightarrow\quad
 \text{common Wallach cone}\\
 \longleftrightarrow\quad
 \text{singular sine-cone limit}.
 \end{gathered}
\]
In particular, the same oscillatory cone exponent that creates the infinite
family also quantizes its focal curvature scale.  The sharp curvature
estimate implies $\sup|\Rm(g_n)|\to\infty$; passing to a subsequence on which
this is strictly increasing gives the pairwise nonisometric family asserted
in Theorem~\ref{thm:main}.  The normalization $\Ric=6g$ fixes the homothety
scale.

\begin{corollary}
The $SU(3)$ Wallach ansatz contains an infinite quantized sequence of
pairwise nonisometric Einstein metrics on $S^7$ exhibiting two-ended
Ricci-flat bubbling.
\end{corollary}

\bigskip

\noindent
\textsc{Anna Siffert}\\
Universit\"at M\"unster, Mathematisches Institut,
Einsteinstr.\ 62, 48149 M\"unster, Germany\\
\texttt{asiffert@uni-muenster.de}

\end{document}